\documentclass{amsart}

\usepackage{amsmath}
\usepackage{lipsum}
\usepackage{amsfonts}
\usepackage{algorithmic}
\usepackage{mathtools}
\usepackage{float}
\usepackage{cleveref}
\usepackage{tikz}
\usepackage{bm}
\usepackage{graphicx}
\usepackage{subfigure}
\usepackage{epstopdf}
\ifpdf
\DeclareGraphicsExtensions{.eps,.pdf,.png,.jpg}
\else
\DeclareGraphicsExtensions{.eps}
\fi

\newtheorem{theorem}{Theorem}[section]
\newtheorem{lemma}[theorem]{Lemma}
\newtheorem{problem}{Problem}[section]
\newtheorem{proposition}{Proposition}[section]

\theoremstyle{definition}

\theoremstyle{remark}
\newtheorem{remark}[theorem]{Remark}

\numberwithin{equation}{section}

\begin{document}

\title{Supersonic flow of Chaplygin gas past a delta wing}

\author{Bingsong Long}
\address{School of Mathematical Sciences, Fudan University, Shanghai, 200433, People's Republic of China}
\email{bslong15@fudan.edu.cn}
\thanks{\textbf{Funding}: This work was supported in part by the Fundamental Research Funds for the Central Universities (2019kfyXJJS134) and China Scholarship Council (201506100083).}

\author{CHAO YI}
\address{Center for Mathematical Sciences, Huazhong University of Science and Technology, Wuhan, 430074, People's Republic of China}
\email{chaoyi@hust.edu.cn}


\date{\today}



\begin{abstract}
We consider the problem of supersonic flow of a Chaplygin gas past a delta wing with a shock or rarefaction wave attached to the leading edges. The flow under study is described by the three-dimensional steady Euler system. In conical coordinates, this problem can be reformulated as a boundary value problem for a nonlinear equation of mixed type. The type of this equation depends fully on the solutions of the problem itself, and thus it cannot be determined in advance. We overcome the difficulty by establishing a crucial Lipschitz estimate, and finally prove the unique existence of the solution via the method of continuity.
\end{abstract}

\maketitle



\section{Introduction}\label{sec:1}
The problem of supersonic flow over delta wings is of great importance in aeronautics, because most supersonic aircraft, like hypersonic planes or missiles, are designed as a triangle or a body having a delta wing (see \cite{AL85}). When a sharp edged delta wing is placed at a small angle of attack in supersonic flow, there arise a shock front on its compression side and a rarefaction wave on its expansion side \cite{RK84}. The shock or the rarefaction wave may be attached to or detached from the leading edges, depending on the Mach number in the flow, the angle of attack and the sweep angle of the wing. During the past decades, many experimental and computational efforts have been made to investigate this problem; see \cite{Baba63,PB69,HH85,Fowe56,Hui71,MW84,Vosk68} and the references therein. However, there has been no rigorous mathematical theory for the global existence of solutions even until now. For the case of a three-dimensional wedge, which can be regarded as the most special delta wing, some related results were announced in \cite{CFa17,Chen92,LXY15}. Under the assumption that the sweep angle is nearly close to zero, the global existence of conical solutions was obtained by Chen-Yi \cite{CY15}. Also, see Chen \cite{Chen97} for a linear approximate solution under the same assumption. In practice, the sweep angle of supersonic aircraft is not that small, so it is still necessary to develop a more general theory.

In this paper, we mainly focus on the study of the above problem for the flow of a Chaplygin gas \cite{Chap04}. The Chaplygin gas is a perfect fluid obeying the following state of equation:
\begin{equation}\label{eq:1.1}
	p(\rho)=A\Big(\frac{1}{\rho_{*}}-\frac{1}{\rho}\Big),
\end{equation}
where $p, \rho>0$ are pressure and density, respectively; $\rho_{*}, A$ are positive constants. As a model of cosmology, the Chaplygin gas can be used to describe the expansion of the universe (see, for instance, \cite{KMP01,Popov10}). It follows from \eqref{eq:1.1} that $\rho{c}=\sqrt{A}$, where $c=c(\rho)$ is the speed of sound. This implies that any shock is reversible and characteristic; see \cite{Serre09,Serre11} for more details. In other words, any rarefaction wave can be treated as a shock but with negative strength. This allows us to discuss the case of rarefaction waves in the same way as the case of shocks. Based on the special properties of the Chaplygin gas, some multidimensional Riemann problems have been well studied; for example, see Serre \cite{Serre09,Serre11}, Chen-Qu \cite{CQ12,CQu12} and Lai-Sheng \cite{SL16}. In what follows, we will first investigate the problem for the special case of a triangular plate, and then in \Cref{sec:4} turn our attention to some thin delta wings with specific shapes. Hereafter, both shocks and rarefaction waves are called pressure waves, and only the attached case is considered.

Now we describe our problem in more details. Let $W_{\sigma}$ denote a flat, infinite-span delta wing in which the angle of apex is $\pi-2\sigma$ with $\sigma\in(0,\pi/2)$. In the rectangular coordinates $(x_1,x_2,x_3)$, it is placed symmetrically on the $x_2Ox_3$-plane, with the apex at the origin and the root chord along the positive $x_3$-axis, that is,
\begin{equation}\label{eq:1.2}
	W_{\sigma}=\{(x_1,x_2,x_3):x_{1}=0, |x_{2}|<x_{3}\cot\sigma, x_{3}>0\}.
\end{equation}
Thus, the sweep angle of $W_{\sigma}$ at the leading edges is just $\sigma$ (see \Cref{fg1}). Throughout the paper, the oncoming flow of uniform state $(\rho_{\infty},q_{\infty})$ is assumed to be supersonic, passing the wing $W_{\sigma}$ at an angle of attack $\alpha$, where $\alpha\in(0,\pi/2)$. Then, the velocity $\bm{v}_{\infty}:=(v_{1\infty},0,v_{3\infty})$ of the oncoming flow is given by $v_{1\infty}=q_{\infty}\sin\alpha$ and $v_{3\infty}=q_{\infty}\cos\alpha$. Let $x_1=s(x_{2},x_{3})$ and $x_1=r(x_{2},x_{3})$ be the equations for the shock and the rarefaction wave attached to the leading edges, respectively. Clearly, both $s$ and $r$ are homogeneous functions of degree one. Write
\begin{align*}
	\mathcal{R}_{\sigma}&:=\{s(x_{2},x_{3})<x_{1}<0,x_{2}>0,x_{3}>0\},\\
	\mathcal{R}_{\sigma}'&:=\{0<x_{1}<r(x_{2},x_{3}),x_{2}>0,x_{3}>0\}.
\end{align*}
By the symmetry of $W_{\sigma}$, it suffices to discuss the problem in the region $\mathcal{R}_{\sigma}$ and $\mathcal{R}_{\sigma}'$. For the attached pressure waves, since the flow field in the compression region and the expansion region are independent, we are allowed to consider the case of shocks and the case of rarefaction waves separately.

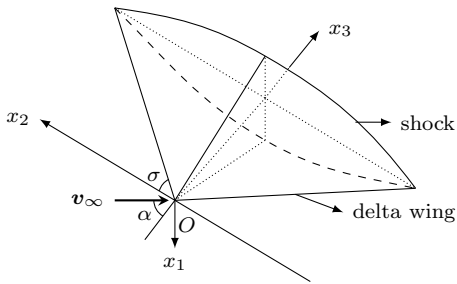
\begin{figure}[H]
	\centering
	\begin{tikzpicture}[smooth, scale=0.8]
	\draw  [-latex] (8.25,1.45)--(3.75,4.15) node [left] {\footnotesize$x_{2}$};
	\draw  [densely dotted] (5,6)--(10,3);
	\draw  (5,6)--(6,2.8)--(10,3);
	\draw  [densely dotted] (7.5,3.8)--(6,2.8)--(7.5,4.5);
	\draw  [densely dotted] (7.5,3.8)--(7.5,5.2);
	\draw  (7.5,5.2)--(6,2.8);
	\draw (5,6) to [out=-8,in=150]  (7.5,5.2) to [out=-30,in=127]  (10,3);
	\draw [dashed] (5,6) to [out=-45,in=150]  (7.5,3.8) to [out=-30,in=-192] (10,3);
	\draw  [-latex] (6,2.8)--(6,2) node [below] {\footnotesize$x_{1}$};
	\node at (6.2,2.4){\footnotesize$O$};
	\draw (6,2.8)--(5.5,2.145);
	\draw  [densely dotted]  (7.5,4.5)--(7.9,5.01);
	\draw  [-latex]  (7.9,5.01)--(8.4,5.625)node [right] {\footnotesize$x_{3}$};
	\draw [thick, -stealth] (5,2.8) node[left, font=\footnotesize] {\footnotesize $\bm{v}_{\infty}$}  --(5.9,2.8) ;
	\draw (5.9,3.15)arc (116:165:0.3);
	\node at (5.65,3.2){\footnotesize$\sigma$};
	\draw (5.65,2.8)arc (180:240:0.3);
	\node at (5.5,2.55){\footnotesize$\alpha$};
	\draw [-latex](8,2.9)--(8.8,2.6) node[right] {\footnotesize{delta wing}};
	\draw [-latex](9,4.1)--(9.6,4.1) node[right] {\footnotesize{shock}};
	\end{tikzpicture}
	\caption{Flat delta wing and attached pressure waves.}
	\label{fg1}
\end{figure}

It is shown in \cite{Serre09,Serre11} that, if a piecewise smooth steady flow of a Chaplygin gas is isentropic and irrotational initially, then it remains so forever. Recall that the oncoming flow has been assumed to be uniform. Hence, the flow under consideration is exactly potential flow. Let us introduce a velocity potential $\Phi$ by $\bm{v}=\nabla_{\bm{x}}\Phi$, where $\bm{x}:=(x_1,x_2,x_3)$. Then the flow is governed by the conservation of mass
\begin{equation}\label{eq:1.3}
	\mathrm{div}_{\bm{x}}(\rho \nabla_{\bm{x}}\Phi)=0
\end{equation}
and the Bernoulli equation
\begin{equation}\label{eq:1.4}
	\frac{1}{2}|\nabla_{\bm{x}}\Phi|^{2}+h(\rho)=\frac12B_{\infty},
\end{equation}
where $\rho$ and $h(\rho)$ are the density and specific enthalpy, respectively; $B_{\infty}/2$ is the Bernoulli constant determined by the oncoming flow, that is,
\begin{equation}\label{eq1:1}
	B_{\infty}=q_{\infty}^2-c_{\infty}^2.
\end{equation}
Combining \eqref{eq:1.1} and \eqref{eq:1.4}, we can express $\rho$ as a function of $\Phi$, i.e.,
\begin{equation}\label{eq:2.29}
	\rho=\frac{\sqrt{A}}{\sqrt{|\nabla_{\bm{x}}\Phi|^{2}-B_{\infty}}}.
\end{equation}
Then we obtain a quasilinear equation for $\Phi$, by substituting \eqref{eq:2.29} into \eqref{eq:1.3}.

Next, we denote by $S_{\sigma}$ and $S_{\sigma}'$ the shock and the rarefaction wave, respectively. Then the Rankine--Hugoniot condition yields
\begin{equation}\label{eq:1.19}
	[\rho \nabla_{\bm{x}}\Phi]\cdot\bm{n}_{s}=0\quad \text{on $S_{\sigma}$},
\end{equation}
where $[\cdot]$ denotes the jump of quantities across the shock, and $\bm{n}_{s}$ is the exterior normal to $S_{\sigma}$. Noting that any shock is a characteristic, the boundary condition \eqref{eq:1.19} is naturally satisfied. Likewise, we have the same conclusion for the rarefaction wave.

Consequently, our problem for the flow of a Chaplygin gas can be formulated mathematically as

\begin{problem}\label{prob1}
	For the wing $W_{\sigma}$ and the oncoming flow given above, we wish to seek a solution $\Phi$ of system \eqref{eq:1.3}--\eqref{eq:1.4} in the region $\mathcal{R}_{\sigma}$ (resp., $\mathcal{R}_{\sigma}'$) with the Dirichlet boundary condition
	\begin{equation}\label{eq1:1.1}
		\Phi=\Phi_{\infty}\quad\text{on $S_{\sigma}$ (resp., $S_{\sigma}'$)}
	\end{equation}
	and the slip boundary conditions
	\begin{align}
		\nabla_{\bm{x}}\Phi\cdot\bm{n}_{w}=0 \quad&\text{on $\{x_1=0\}$},\label{eq2:1.2}\\
		\nabla_{\bm{x}}\Phi\cdot\bm{n}_{sy}=0\quad &\text{on $\{x_2=0\}$},\label{eq1:1.2}
	\end{align}
	where $\Phi_{\infty}=v_{1\infty}x_{1}+v_{3\infty}x_{3}$ is the potential of the oncoming flow; $\bm{n}_{w}=(1,0,0)$ is the exterior normal to $\{x_1=0\}$, and $\bm{n}_{sy}=(0,-1,0)$ is the exterior normal to $\{x_2=0\}$.
\end{problem}

The following theorem is the main result of this paper.

\begin{theorem}[Main Theorem]\label{thm: Main}
	Assume that the state $(\rho_{\infty},q_{\infty})$ of the oncoming flow is uniform and supersonic, and the wing $W_{\sigma}$ is a triangular plate given by \eqref{eq:1.2}. Then, for the case of shocks, we can find a critical angle $\alpha_{0}=\alpha_{0}(\rho_{\infty},q_{\infty})\in(0,{\pi}/{2})$ so that for any fixed $\alpha\in(0,\alpha_{0})$, there exists $\sigma_{0}=\sigma_{0}(\rho_{\infty},q_{\infty},\alpha)\in(0,{\pi}/{2})$ such that, when $\sigma\in[0,\sigma_{0}]$, \cref{prob1} admits a piecewise smooth solution.
	
	Similarly, by replacing the angle $\alpha_{0}$ with $\alpha'_{0}\in(0,{\pi}/{2})$, the angle $\sigma_{0}$ with $\sigma'_{0}\in(0,\sigma_0)$, and the condition $\sigma\in[0,\sigma_{0}]$ with $\sigma\in[0,\sigma'_{0})$, we obtain the same result for the case of rarefaction waves.
\end{theorem}

Although \cref{prob1} is three dimensional in nature, it can be treated mathematically as a two-dimensional one, due to the features of the wing $W_{\sigma}$ and the resulting conical flow. In fact, we can view the scaled variables $\xi_1=x_1/x_3$ and $\xi_2=x_2/x_3$ as new coordinates, and then restate \cref{prob1} in the $(\xi_{1},\xi_{2})$ coordinates, as demonstrated in \Cref{sec:2.3}. This finally leads to a boundary value problem for a nonlinear mixed-type equation; see \cref{prob2}. Unlike those discussed in \cite{CFe10,CY15}, the type of our equation is far from being known, due to the fact that the sweep angle of $W_{\sigma}$ is no longer sufficiently small. Moreover, from \eqref{eq:1.1} and the equation of state for a polytropic gas, namely, $p(\rho)=A\rho^{\gamma}$ with constants $A,\gamma>0$, we obtain $\gamma=-1$ for the Chaplygin gas. This means that the ellipticity principle proved by Elling-Liu \cite{EL05} for self-similar potential flow cannot be applied to our problem either, because the approach adopted there is valid only for $\gamma>-1$. Motivated by Serre \cite{Serre09,Serre11}, we find that the type of our equation can be determined completely by a prior estimates for the solutions of \cref{prob2}. It should be pointed out that, our estimates can be established under a weaker condition that may allow the nonlinear equation degenerate inside the domain. By contrast, the corresponding equation in \cite{Serre11} is assumed to be elliptic over the whole domain. We also note that in \cite{Serre11} only the Dirichlet problem has been studied. In \Cref{sec:4}, with some new ingredients added into the strategy mentioned above, we are able to treat a more general class of boundary value problems, such as those with mixed boundary conditions or in a Lipschitz domain.

The rest of the paper is arranged as follows. \Cref{sec:2} mainly presents some preliminaries to the investigation of \cref{prob1}. We first determine the undisturbed states of the downstream flow near the leading edges of the wing, and then analyze the global structures of the pressure waves in conical coordinates. After that, we reduce \cref{prob1} to a boundary value problem for a nonlinear mixed-type equation, that is, \cref{prob2}, and meanwhile we rewrite the main theorem as \Cref{thm4}. \Cref{sec:3} is devoted to the proof of \Cref{thm4}. With the use of two key auxiliary functions, we establish a prior estimates for the solutions of \cref{prob2}, which ensures the ellipticity of the equation under consideration, and therefore allows us to obtain the existence and uniqueness for \cref{prob2}. \Cref{sec:4} considers the problem of supersonic flow over a thin delta wing of diamond cross-sections. We achieve a similar result by solving a Neumann boundary value problem in a Lipschitz domain. \Cref{sec:5} gives a brief discussion of the difficulties that arise from the non-convexity of the domains in the study of this problem.

\section{Preliminary analysis of \cref{prob1}}\label{sec:2}
When the wing $W_{\sigma}$ becomes a half-plane, i.e., $\sigma=0$, \cref{prob1} is essentially a two-dimensional problem, so it can be solved by the analysis of shock polars, as shown in \Cref{sec:appendix a}. From now on, we only consider the case $\sigma>0$.
\subsection{Uniform downstream flow near the leading edges}\label{sec:2.1}
Let us begin with the potential equation \eqref{eq:1.3}. Expanding \eqref{eq:1.3}, together with \eqref{eq:2.29}, we obtain a quasilinear equation of second order
\begin{multline}\label{eq:2.30}
	(c^{2}-\Phi^{2}_{x_{1}})\Phi_{x_{1}x_{1}}+(c^{2}-\Phi^{2}_{x_{2}})\Phi_{x_{2}x_{2}}+(c^{2}-\Phi^{2}_{x_{3}})\Phi_{x_{3}x_{3}}\\-2\Phi_{x_{1}}\Phi_{x_{2}}\Phi_{x_{1}x_{2}}-2\Phi_{x_{1}}\Phi_{x_{3}}\Phi_{x_{1}x_{3}}-2\Phi_{x_{2}}\Phi_{x_{3}}\Phi_{x_{2}x_{3}}=0.
\end{multline}
The characteristic equation of $\eqref{eq:2.30}$ is
\begin{equation}\label{eq:2.31}
	Q(\bm{\zeta})=c^{2}-|\nabla_{\bm{x}}\Phi\cdot\bm{\zeta}|^{2} \quad\text{for any}~\bm{\zeta}\in\mathbb{R}^{3},~|\bm{\zeta}|=1,
\end{equation}
so equation \eqref{eq:2.30} is elliptic in a subsonic domain and hyperbolic in a supersonic domain. Also, since any pressure wave is a characteristic, the normal component of the flow velocity across the pressure waves is sonic. Then, if a stream of flow is initially supersonic, it stays so forever. This means that the downstream flow is supersonic, under the assumption that the oncoming flow is supersonic. Accordingly, equation \eqref{eq:2.30} is hyperbolic and there exists a Mach cone of the apex of the wing. Thanks to the property of finite propagation for hyperbolic equations, the solution of $\eqref{eq:2.30}$ outside the Mach cone is undisturbed, so it can be analyzed in the same way as that for supersonic flow past a wedge. In particular, for a uniform state of the downstream flow, the pressure wave is flat and attached to the leading edges of the wing. For simplicity of notation, we use $S_{ob}$ and $S'_{ob}$, respectively, to denote the flat shock and flat rarefaction wave.

We then calculate the solution of $\eqref{eq:2.30}$ outside the Mach cone. Let us first consider the case of shocks. Since the oncoming flow is not perpendicular to the leading ledges, we introduce
\begin{equation*}
	\bm{e}_1:=(1,0,0), \quad \bm{e}_i:=(0,\cos\sigma,\sin\sigma),\quad  \bm{e}_j:=(0,-\sin\sigma,\cos\sigma).
\end{equation*}
Obviously, $\{\bm{e}_1,\bm{e}_i,\bm{e}_j\}$ is an orthogonal basis for a system of coordinates $(x_1,x_i,x_j)$. Then the velocity $\bm{v}_{\infty}$ is decomposed as
\begin{equation}\label{eq:2.1}
	\bm{v}_{\infty}=v_{1\infty}\bm{e}_1+v_{3\infty}\sin\sigma\bm{e}_i+v_{3\infty}\cos\sigma\bm{e}_j.
\end{equation}
For notational convenience, we write
\begin{equation*}
	\tilde{\bm{v}}_{\infty}=v_{1\infty}\bm{e}_1+v_{3\infty}\cos\sigma\bm{e}_j,
\end{equation*}
thus $\tilde{q}_{\infty}=\sqrt{v^2_{1\infty}+v^2_{3\infty}\cos^2\sigma}$ and $\alpha_{n}=\arctan({\tan\alpha}/{\cos\sigma})$, where $\alpha_{n}$ is the angle between $\tilde{\bm{v}}_{\infty}$ and the $x_{2}Ox_{3}$-plane. Let $\bm{v}_{\sigma}=(0, v_{2\sigma}, v_{3\sigma})$ be the velocity of the uniform flow behind the shock, and $q_{j\sigma}$ the speed of the flow along $\bm{e}_j$. Since the velocity of the flow along $\bm{e}_i$ is unchanged across the flat shock $S_{ob}$, we have
\begin{equation*}
	\bm{v}_{\sigma}=v_{3\infty}\sin\sigma\bm{e}_i+q_{j\sigma}\bm{e}_j,
\end{equation*}
which implies
\begin{equation}\label{eq1:1.3}
	v_{2\sigma}=v_{3\infty}\sin\sigma\cos\sigma-q_{j\sigma}\sin\sigma,\quad v_{3\sigma}=v_{3\infty}\sin^2\sigma+q_{j\sigma}\cos\sigma.
\end{equation}
Notice that $q_{j\sigma}$ can be derived by \eqref{eq:A.3} with the choice $u_{0}=\tilde{q}_{\infty}$, $c_{0}=c_\infty$ and $\alpha=\alpha_{n}$. From \eqref{eq1:1.3}, we obtain the explicit expression of $\bm{v}_{\sigma}$. So the solution of equation $\eqref{eq:2.30}$ outside the Mach cone is $\Phi_{\sigma}=v_{2\sigma}x_2+v_{3\sigma}x_3$. In addition, we get the corresponding sound speed $c_{\sigma}=\sqrt{|\nabla_{\bm{x}}\Phi_{\sigma}|^{2}-B_{\infty}}$, where $B_{\infty}$ is given by \eqref{eq1:1}. Analogously, we can obtain a uniform state $(c'_{\sigma},(0, v'_{2\sigma}, v'_{3\sigma}))$ and a potential function $\Phi'_{\sigma}=v_{2\sigma}'x_2+v_{3\sigma}'x_3$ for the case of rarefaction waves.

Now we turn to the role of the angles $\alpha$ and $\sigma$. From \Cref{sec:appendix a}, we see that for the Chaplygin gas, there may be a phenomenon of concentration or cavitation, if the angle of the wedge changes excessively. To avoid this, we restrict the ranges of $\alpha$ and $\sigma$. Let $\beta_{n}$ be the angle between the flat shock $S_{ob}$ and the $x_{2}Ox_{3}$-plane, and $\beta'_{n}$ the angle between the flat rarefaction wave $S'_{ob}$ and the $x_{2}Ox_{3}$-plane. Then from the shock polar given in \Cref{sec:appendix a}, we know that the circle with center  $O_{\infty}(\tilde{q}_{\infty}\cos\alpha_{n},\tilde{q}_{\infty}\sin\alpha_{n})$ and radius $c_\infty$ is tangent to $S_{ob}$ at point $P$, and also tangent to $S'_{ob}$ at point $P'$ (see \Cref{fg7}). Note that, in the $x_1Ox_j$-plane, both $\tilde{q}_{\infty}$ and $c_\infty$ are invariant under any rotation transformation. We here apply the conclusion in \Cref{sec:appendix a} directly.

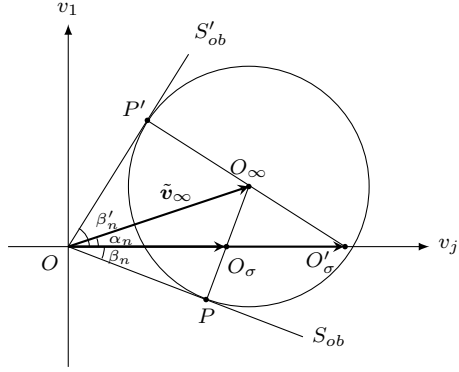
\begin{figure}[H]
	\centering
	\begin{tikzpicture}[smooth, scale=0.8]
	\draw [-latex](2,6)--(9,6)node[right]{\footnotesize$v_j$};
	\draw [-latex](3,4)--(3,9.7)node[above]{\footnotesize$v_1$};
	\draw (6,7) circle [radius=2];
	\draw [thick,-stealth](3,6)--(6,7)node[above]{\footnotesize$O_{\infty}$};
	\node at(4.8,6.6) [above]{\footnotesize$\tilde{\bm{v}}_{\infty}$};
	\draw (6,7)--(5.3,5.12)node[below]{\footnotesize$P$};
	\draw [thick,-stealth] (3,6)--(5.6,6);
	\node at (5.9,5.95)[below]{\footnotesize$O_{\sigma}$};	
	\draw [thick,-stealth] (3,6)--(7.58,6);
	\node at(7.2,5.75){\footnotesize$O_{\sigma}'$};
	\draw (3,6)--(5.6,5)--(6.9,4.5)node[right]{\footnotesize$S_{ob}$};
	\node at (3,6)[below left]{\footnotesize$O$};
	\draw (3.5,6) arc (0:20:0.4);
	\draw (3.6,6) arc (0:-23:0.5);
	\node at (3.52,6.1)[right]{\tiny$\alpha_{n}$};
	\node at (3.5,5.82)[right]{\tiny$\beta_{n}$};
	\draw (3,6)--(5,9.2);
	\node at (5.8,9.9)[below left]{\footnotesize$S_{ob}'$};
	\node at (4.1,8.3){\footnotesize$P'$};
	\node at (3.3,6.45)[right]{\tiny$\beta_{n}'$};
	\draw (4.31,8.1) --(7.6,6);
	\draw (3.35,6) arc (0:60:0.35);
	\fill (5.63,6) circle (1.3pt);
	\fill (6,7) circle(1.3pt);
	\fill (7.6,6) circle(1.3pt);
	\fill (4.315,8.1) circle(1.3pt);
	\fill (5.29,5.12) circle(1.3pt);
	\end{tikzpicture}
	\caption{Determination of flat pressure waves.}
	\label{fg7}
\end{figure}

To avoid concentration, we deduce from \eqref{eq:A.4} that
\begin{equation}\label{eq:2.2}
	c_{\infty}<\tilde{q}_{\infty}<\frac{c_{\infty}}{\sin\alpha_{n}}.
\end{equation}
Owing to the relation $\tan\alpha={v_{1\infty}}/{v_{3\infty}}$, the speed $\tilde{q}_{\infty}$ equals
\begin{equation}\label{eq:2.4}
	\tilde{q}_{\infty}=v_{1\infty}\sqrt{1+\frac{\cos^{2}\sigma}{\tan^{2}\alpha}}.
\end{equation}
Substituting \eqref{eq:2.4} into the right-hand side of inequality \eqref{eq:2.2}, we get
\begin{equation*}\label{eq:2.5}
	q_{\infty}\sin\alpha=v_{1\infty}<c_{\infty},
\end{equation*}
which leads to
\begin{equation}\label{eq:2.6}
	\alpha_{0}:=\arcsin\Big(\frac{c_{\infty}}{q_{\infty}}\Big)>\alpha.
\end{equation}
In addition, the left-hand side of inequality \eqref{eq:2.2} can be reduced to
\begin{equation}\label{eq:2.7}
	\sigma_0:=\arcsin\Big(\frac{\sqrt{q^{2}_{\infty}-c^{2}_{\infty}}}{v_{3\infty}}\Big)>\sigma.
\end{equation}

Then, to avoid cavitation, it follows from \eqref{eq:A.5} that
\begin{equation}\label{eq:2.8}
	\tilde{q}_{\infty}>\frac{c_{\infty}}{\cos\alpha_{n}}.
\end{equation}
Since the speed $\tilde{q}_{\infty}$ also takes the form
\begin{equation}\label{eq:2.9}
	\tilde{q}_{\infty}=v_{3\infty}\cos\sigma\sqrt{1+\frac{\tan^{2}\alpha}{\cos^{2}\sigma}},
\end{equation}
combining \eqref{eq:2.8} and \eqref{eq:2.9}, we have
\begin{equation}\label{eq:2.10}
	v_{3\infty}>\frac{c_{\infty}}{\cos\sigma}.
\end{equation}
Then inserting $\sigma=0$ into \eqref{eq:2.10} gives
\begin{equation}\label{eq:2.11}
	\alpha_{0}':=\arccos\Big(\frac{c_{\infty}}{q_{\infty}}\Big)>\alpha.
\end{equation}
Moreover, we infer from \eqref{eq:2.10} that
\begin{equation}\label{eq1:2.1}
	\sigma_{0}':=\arccos\Big(\frac{c_{\infty}}{v_{3\infty}}\Big)>\sigma.
\end{equation}
Obviously, $\sigma_{0}'<\sigma_{0}$ by \eqref{eq:2.7} and \eqref{eq1:2.1}.
\begin{remark}\label{remark:1.1}
	For the case of shocks, if we fix $\alpha\in (0,\alpha_0)$ and let $\sigma$ vary in $(0,\sigma_0)$, then from
	\begin{equation*}
		\tilde{q}_{\infty}\sin\alpha_{n}={q}_{\infty}\sin\alpha,\quad\tilde{q}_{\infty}\cos\alpha_{n}={q}_{\infty}\cos\alpha\cos\sigma,
	\end{equation*}
	we see that the $v_1$-coordinate of $O_\infty$ remains unchanged. Note that the radius of the circle in \Cref{fg7} is always $c_\infty$. This implies that the angle $\beta_{n}$ is a monotonically increasing function of $\sigma$. Since the phenomenon of concentration occurs only when $\beta_{n}=0$, it would never occur unless it happened for $\sigma=0$.
	
	However, for the case of rarefaction waves, the phenomenon of cavitation could really occur as long as we fix $\alpha\in (0,\alpha'_0)$ and let $\sigma$ approaches $\sigma'_0$. This can be observed from \Cref{fg7} that, since the center $O_\infty$ moves left as $\sigma$ increases, the circle finally touches the $v_1$-axis when $\sigma=\sigma'_0$.
\end{remark}

\subsection{Pressure wave patterns in conical coordinates}\label{sec:2.2}
As mentioned in \Cref{sec:1}, we can treat the case of shocks and the case of rarefaction waves separately. In what follows, we fix $\alpha\in (0,\alpha_0)$ for the case of shocks, and $\alpha\in (0,\alpha'_0)$ for the case of rarefaction waves, where $\alpha_0$ and $\alpha'_0$ are given by \eqref{eq:2.6} and \eqref{eq:2.11}, respectively. We will show that the attached shock appears for $\sigma\in(0,\sigma_0]$, while the attached rarefaction wave appears only for $\sigma\in(0,\sigma'_0)$. Moreover, the patterns of these waves will be demonstrated explicitly in a rectangular system of conical coordinates, as defined below.

Notice that the boundary value problem \eqref{eq:1.3}--\eqref{eq:1.4} and \eqref{eq1:1.1}--\eqref{eq1:1.2} is invariant under the scaling
\begin{equation*}
	\bm{x}\longrightarrow \varsigma\bm{x}, \quad 	(\rho,\Phi)\longrightarrow\Big(\rho,\dfrac{\Phi}{\varsigma}\Big)\quad\quad \text{for} \quad\varsigma\neq 0.
\end{equation*}
Thus, we seek a solution with the following form:
\begin{equation}\label{eq:2.32}
	\rho(\bm{x})=\rho(\xi_{1},\xi_{2}), \quad \Phi(\bm{x})=x_3\phi(\xi_{1},\xi_{2}),
\end{equation}
where $(\xi_{1},\xi_{2}):=({x_1}/{x_3},{x_2}/{x_3})$ are called conical coordinates. Hereafter, our discussion is carried out in this coordinates, along with the notations
\begin{equation}\label{eq1:2.12}
	\psi:=\frac{\phi}{\sqrt{B_{\infty}}},\quad a:=\frac{c}{\sqrt{B_{\infty}}},
\end{equation}
where the positive constant $B_\infty$ is defined by \eqref{eq1:1}.

Before proceeding, we introduce some notations. Let $\mathcal{C}_{\infty}$, $\mathcal{C}_{\sigma}$ and $\mathcal{C}'_{\sigma}$ be the Mach cones of the apex of the wing, determined by the oncoming flow, the flow behind the shock, and the flow behind the rarefaction wave, respectively. By abuse of notation but without misunderstanding, we continue to write $\mathcal{C}_{\infty}$, $\mathcal{C}_{\sigma}$, and $\mathcal{C}'_{\sigma}$ for the corresponding curves of the Mach cones, $S_{ob}$ and $S_{ob}'$ for the corresponding oblique shock and oblique rarefaction wave, respectively, in the conical coordinates.

Let us first derive the equations for $S_{ob}$ and $S_{ob}'$. By the continuity of $\Phi$ on the flat pressure waves, together with
\eqref{eq:2.32}, we have $\psi_\infty=\psi_\sigma$ on $S_{ob}$ and $\psi_\infty=\psi_\sigma'$ on $S_{ob}'$.
Moreover, using \eqref{eq:2.32}--\eqref{eq1:2.12} and the explicit expressions of $\Phi_\infty$, $\Phi_\sigma$ and $\Phi_\sigma'$, we get
\begin{equation}\label{eq1:1.5}
	\psi_\infty=\frac{v_{1\infty}\xi_1+v_{3\infty}}{\sqrt{B_{\infty}}},\quad
	\psi_\sigma=\frac{v_{2\sigma}\xi_2+v_{3\sigma}}{\sqrt{B_{\infty}}},\quad
	\psi'_\sigma=\frac{v'_{2\sigma}\xi_2+v'_{3\sigma}}{\sqrt{B_{\infty}}}.
\end{equation}
Then from \eqref{eq1:1.5}, it follows that the equations for $S_{ob}$ and $S_{ob}'$ are given by
\begin{align}
	&S_{ob}:\quad v_{1\infty}\xi_1+v_{3\infty}=v_{2\sigma}\xi_2+v_{3\sigma},\label{eq1:1.10}\\
	&S_{ob}':\quad v_{1\infty}\xi_1+v_{3\infty}=v'_{2\sigma}\xi_2+v'_{3\sigma}.\label{eq1:1.11}
\end{align}

We then turn to the equations for $\mathcal{C}_{\infty}$, $\mathcal{C}_{\sigma}$, and $\mathcal{C}'_{\sigma}$. It follows from \eqref{eq:B.6} and \eqref{eq:2.32}--\eqref{eq1:2.12} that, in the conical coordinates, the equation for a Mach cone of the apex of the wing takes the form
\begin{equation}\label{eq:2.33}
	|D\psi|^{2}+|\psi-D\psi\cdot\bm{\xi}|^{2}-\frac{\psi^{2}}{1+|\bm{\xi}|^{2}}=a^{2},
\end{equation}
where $\bm{\xi}:=(\xi_1,\xi_2)$. By \eqref{eq:2.29} and \eqref{eq:2.32}--\eqref{eq1:2.12}, and noting $\rho c=\sqrt{A}$, we have
\begin{equation}\label{eq:2.34}
	a^2=|D\psi|^{2}+|\psi-D\psi\cdot\bm{\xi}|^{2}-1.
\end{equation}
Substituting \eqref{eq:2.34} into \eqref{eq:2.33} yields
\begin{equation}\label{eq:2.35}
	\psi^2 = 1+|\bm{\xi}|^2.
\end{equation}
Plugging \eqref{eq1:1.5} into \eqref{eq:2.35}, we obtain the following equations:
\begin{alignat}{2}
	\mathcal{C}_{\infty}&:\quad (v_{1\infty}\xi_1+v_{3\infty})^2&&=B_{\infty}(1+|\bm{\xi}|^2),\label{eq1:1.7}\\
	\mathcal{C}_{\sigma}&:\qquad
	(v_{2\sigma}\xi_2+v_{3\sigma})^2&&=B_{\infty}(1+|\bm{\xi}|^2),\label{eq1:1.8}\\
	\mathcal{C}'_{\sigma}&:\qquad
	(v'_{2\sigma}\xi_2+v'_{3\sigma})^2&&=B_{\infty}(1+|\bm{\xi}|^2).\label{eq1:1.9}
\end{alignat}

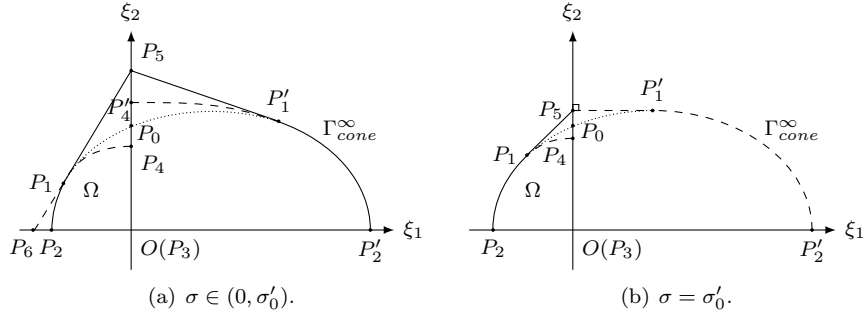
\begin{figure}[H]
	\centering
	\hspace{0.45cm}
	\subfigure[$\sigma\in(0,\sigma_{0}')$.]{
		\begin{tikzpicture}[smooth, scale=0.53]
		\draw [-latex](1.2,2)--(10.6,2)node[right]{\footnotesize$\xi_{1}$};
		\draw [-latex](4,1)--(4,7)node[above]{\footnotesize$\xi_{2}$};
		\node [below right] at(4,2){\footnotesize $O(P_{3})$};
		\draw (10,2) arc (0:65:4 and 3);
		\draw [densely dotted] (7.7,4.715) arc(65:160:3.98 and 2.8);
		\draw (2.3,3.15) arc(160:180:4.8 and 3.4);
		\node at (2,2)[below]{\footnotesize $P_{2}$};
		\node at (10,2)[below]{\footnotesize $P'_{2}$};
		\draw (4,6)node[above right]{\footnotesize $P_{5}$} -- (7.7,4.73)node[above]{\footnotesize $P'_{1}$};
		\draw [dashed] (4,5.2) to [out=0, in=160] (7.7,4.715);
		\node at(4.25,5.03)[left]{\footnotesize $P'_{4}$};
		\draw (4,6) -- (2.3,3.17)node[left]{\footnotesize $P_{1}$};
		\draw [dashed] (2.3,3.17)--(1.575,2);
		\node at (1.26,1.55){\footnotesize $P_{6}$};
		\draw [dashed] (4,4.1)node[below right]{\footnotesize $P_{4}$} to [out=180, in=65] (2.3,3.17);
		\node at(3,3){\footnotesize$\Omega$};
		\node at(9.5,4.5){\footnotesize$\Gamma_{cone}^{\infty}$};
		\node at(3.9,4.4)[right]{\footnotesize $P_{0}$};
		\fill (2.3,3.17)circle(1.3pt);
		\fill (4,4.62)circle(1.3pt);
		\fill (4,4.1)circle(1.3pt);
		\fill (4,5.2)circle(1.3pt);
		\fill (4,6)circle(1.3pt);
		\fill (7.7,4.73)circle(1.3pt);
		\fill (2,2)circle(1.3pt);
		\fill (10,2)circle(1.3pt);
		\fill (1.54,2)circle(1.3pt);
		\end{tikzpicture}
		\label{fg11}
	}
	\subfigure[$\sigma=\sigma_{0}'$.]{
		\begin{tikzpicture}[smooth, scale=0.53]
		\draw [-latex](1.4,2)--(10.6,2)node[right]{\footnotesize$\xi_{1}$};
		\draw [-latex](4,1)--(4,7)node[above]{\footnotesize$\xi_{2}$};
		\node [below right] at(4,2){\footnotesize $O(P_{3})$};
		\draw [dashed] (10,2) arc (0:90:4 and 3);
		\draw [densely dotted] (6.05,5) arc(90:140:3.98 and 2.8);
		\draw (2.93,3.94) arc(140:180:4 and 3);
		\node at (2,2)[below]{\footnotesize $P_{2}$};
		\node at (10,2)[below]{\footnotesize $P'_{2}$};
		\draw [dashed] (4,5)node[left]{\footnotesize $P_{5}$} -- (6,5)node[above]{\footnotesize $P'_{1}$};
		\draw (4,5) rectangle (4.15,5.15);
		\draw (4,5)--(2.85,3.88)node[left]{\footnotesize $P_{1}$};
		\draw [dashed] (4,4.3) to [out=180, in=40] (2.85,3.88);
		\node[below left] at(4.1,4.3){\footnotesize $P_{4}$};
		\node at(3,3){\footnotesize$\Omega$};
		\node at(9.5,4.5){\footnotesize$\Gamma_{cone}^{\infty}$};
		\node at(3.95,4.45)[right]{\footnotesize $P_{0}$};
		\fill (2.85,3.88)circle(1.3pt);
		\fill (4,4.62)circle(1.3pt);
		\fill (4,4.3)circle(1.3pt);
		\fill (4,5)circle(1.3pt);
		\fill (6,5)circle(1.3pt);
		\fill (2,2)circle(1.3pt);
		\fill (10,2)circle(1.3pt);
		\end{tikzpicture}
		\label{fg11'}
	}	
	\caption{Patterns of pressure waves in the $(\xi_1,\xi_2)$-plane.}
\end{figure}

Now we are ready to analyze the global structures of pressure waves. Since any shock is a characteristic, the oblique shock $S_{ob}$ must be tangent to the curve $\mathcal{C}_{\infty}$ at a point, denoted by $P_1$. This is also true for the case of rarefaction waves, except for the tangent point, denoted by $P_1'$. Also, we denote by $P_0$ the intersection point of $\mathcal{C}_{\infty}$ and the $\xi_2$-axis, by $P_2$ the intersection point of $\mathcal{C}_{\infty}$ and the negative $\xi_1$-axis, by $P_4$ (resp., $P_4'$) the intersection point of $\mathcal{C}_{\sigma}$ (resp., $\mathcal{C}_{\sigma}'$) and the $\xi_2$-axis, and by $P_5$ the intersection point of the oblique shocks and the $\xi_2$-axis (see \Cref{fg11}). In addition, by using \eqref{eq1:1.10} and \eqref{eq1:1.7}, we have $P_5(0,\cot\sigma)$ and $P_0(0,\sqrt{c^2_\infty-v^2_{1\infty}}/\sqrt{B_\infty})$. It follows from \eqref{eq:2.7} that when $\sigma\in(0,\sigma_{0})$, the point $P_5$ is above $P_0$ all the time.
Then, we establish the relationship between the curve $\mathcal{C}_{\infty}$ and the curve $\mathcal{C}_{\sigma}$ (resp., $\mathcal{C}_{\sigma}'$) as follows.

\begin{lemma}\label{lemma2}
	Let $\mathcal{C}_{\infty}$, $\mathcal{C}_{\sigma}$ and $\mathcal{C}'_{\sigma}$ be defined as \eqref{eq1:1.7}--\eqref{eq1:1.9}. For any fixed $\alpha\in (0,\alpha_0)$, if $\sigma\in(0,\sigma_{0})$, then the curve $\mathcal{C}_{\sigma}$ is tangent to $\mathcal{C}_{\infty}$ only at the point $P_1$, and the point $P_4$ is always below $P_0$, where $\alpha_0$ and $\sigma_{0}$ are given by \eqref{eq:2.6} and \eqref{eq:2.7}, respectively.
	
	With $\alpha_0$, $\sigma_{0}$, $P_1$ and $P_4$ replaced, respectively, by $\alpha_0'$, $\sigma_{0}'$, $P'_1$ and $P_4'$, we obtain the same result for $\mathcal{C}_{\infty}$ and $\mathcal{C}'_{\sigma}$ except that the point $P_4'$ is located between $P_0$ and $P_5$, where $\alpha'_0$ and $\sigma'_{0}$ are given by \eqref{eq:2.11} and \eqref{eq1:2.1}, respectively.
\end{lemma}

\begin{proof}
	Let us first prove that $\mathcal{C}_{\infty}$ and $\mathcal{C}_{\sigma}$ are tangent at the point $P_1$. It follows from \eqref{eq1:1.7} and \eqref{eq1:1.8} that the intersection points of $\mathcal{C}_{\infty}$ and $\mathcal{C}_{\sigma}$ satisfy
	\begin{equation}\label{eq:2.3}
		|v_{1\infty}\xi_1+v_{3\infty}|=|v_{2\sigma}\xi_2+v_{3\sigma}|.
	\end{equation}
	We claim that the function $\psi_\infty$ is positive on the curve $\mathcal{C}_{\infty}$. The explicit expression of $\psi_\infty$ implies $\psi_{\infty}>0$ on the arc $\mathcal{C}_{\infty}\cap \{\xi_2\geq0,\xi_1\geq0\}$. Then it suffices to consider the function $\psi_\infty$ on the remaining part of $\mathcal{C}_{\infty}\cap \{\xi_2\geq0,\xi_1<0\}$. Denoted by $P_6$ the intersection point of the extension line of $S_{ob}$ and the $\xi_1$-axis. Since the oblique shock $S_{ob}$ is tangent to the curve $\mathcal{C}_{\infty}$, the point $P_6$ must lie in the left-hand side of $P_2$ (see \Cref{fg11}). From \eqref{eq1:1.5}--\eqref{eq1:1.10} and the explicit expression of $\psi_\infty$ and $S_{ob}$, it follows that in the triangle $\Delta P_3P_5P_6$, the following inequality holds
	\begin{equation*}
		\psi_{\infty}(\bm{\xi})\geq\frac{v_{1\infty}\xi_{P_6}+v_{3\infty}}{\sqrt{B_{\infty}}}=\frac{v_{3\sigma}}{\sqrt{B_{\infty}}}>0,
	\end{equation*}
	where $\xi_{P_6}$ denotes the $\xi_1$-coordinate of $P_6$. In addition, it is clear that $\mathcal{C}_{\infty}\cap \{\xi_2\geq0,\xi_1<0\}\subset \Delta P_3P_5P_6$. Thus we have shown the positivity of $\psi_\infty$. Moreover, in this conical coordinates, we know $\psi_\sigma>0$ on the curve $\mathcal{C}_{\sigma}$. Therefore, equality \eqref{eq:2.3} is equivalent to
	\begin{equation}\label{eq:2.49}
		v_{1\infty}\xi_1+v_{3\infty}=v_{2\sigma}\xi_2+v_{3\sigma}.
	\end{equation}
	In fact, \eqref{eq:2.49} is the equation for the oblique shock $S_{ob}$.
	Also, since the curve $\mathcal{C}_{\infty}$ is tangent to $S_{ob}$ at the point $P_1$, it follows that there is only one intersection point of $\mathcal{C}_{\infty}$ and $\mathcal{C}_{\sigma}$, and moreover they are tangent at the point $P_1$.
	
	Then we prove that for any fixed $\alpha\in (0,\alpha_0)$, the point $P_{4}$ is always below $P_{0}$ when $\sigma\in(0,\sigma_0)$. Noting that the point $P_5$ is above $P_0$ when $\sigma\in(0,\sigma_0)$,  then $P_1$ is the only intersection point of $\mathcal{C}_{\infty}$ and $\mathcal{C}_{\sigma}$ from the above discussion. We also know that  the point $P_{4}$ is below $P_{0}$ when $\sigma=0$. Then for $\sigma\in(0,\sigma_0)$,  the point $P_{4}$ is below $P_{0}$, as the coordinate of $P_4$ is a continuous function of $\sigma$.
	
	As for the point $P'_4$, since the curve $\mathcal{C}'_{\sigma}$ is tangent to the oblique rarefaction wave $S'_{ob}$, it is always below $P_5$. Also, noting that the point $P'_4$ is above $P_{0}$ for $\sigma=0$, we deduce that the point $P'_4$ must be located between $P_0$ and $P_5$. The proof is completed.
\end{proof}

\begin{remark}\label{remark:1.2}
	We explain here that the oblique rarefaction wave $S_{ob}'$ is perpendicular to the $\xi_2$-axis when $\sigma=\sigma_{0}'$.  Note that the location of $\mathcal{C}_{\infty}$ is known (see the curve $P_2P_0P_2'$ in \Cref{fg11'}). Then for any fixed $\alpha\in (0,\alpha_0)$, as $\sigma$ varies from zero to $\sigma_0$, there must exist a critical angle such that $S_{ob}'$ is perpendicular to the $\xi_2$-axis. Also, we know from \Cref{sec:appendix a} that if the angle between the flat rarefaction wave and the velocity of the uniform flow behind $S'_{ob}$ approaches to $\pi/2$, then the phenomenon of cavitation will occur. Therefore, we infer from \eqref{eq:2.10} and \eqref{eq1:2.1} that $\sigma=\sigma_{0}'$ is the critical angle.
\end{remark}

With the above analysis, we are able to draw the patterns of pressure waves as in \Cref{fg11}.  Note that the points $P_{4}$ and $P_{5}$ meet at $P_0$ when $\sigma=\sigma_{0}$. Consequently, the vertex angle of the wing $W_\sigma$ is greater than the apex angle of the Mach cone $\mathcal{C}_\sigma$ until $\sigma=\sigma_0$. This implies that the attached shock occurs only for $\sigma\in(0,\sigma_0]$. Furthermore, we conclude from \Cref{lemma2} and \Cref{remark:1.2} that there is a rarefaction wave attached to the leading edge for $\sigma\in(0,\sigma_{0}')$. Since the discussion of the case of rarefaction waves is similar to that of the case of shocks, we mainly consider the case of shocks afterwards.

\begin{figure}[H]
	\centering
	\subfigure[$\sigma\in(0,\sigma_{0})$.]{
		\begin{tikzpicture}[smooth, scale=0.7]
		\draw  [-latex](1,2)--(5.8,2) node [right] {\footnotesize$\xi_{1}$};
		\draw  [-latex](4,1)--(4,8) node [above] {\footnotesize$\xi_{2}$};
		\draw  (2,2)   to [out=90, in =245] (2.78,5.45) node [above left] {\footnotesize$P_{1}$} ;
		\draw [densely dotted] (4,7) node [right] {\footnotesize$P_{0}$}to [out=223, in =233] (2.82,5.42) ;
		\draw [dashed] (4, 6.1)node [right] {\footnotesize$P_{4}$} to [out=180, in =57] (2.82,5.42) ;
		\node  at (2,2) [below] {\footnotesize$P_{2}$};
		\node  at (4,2) [below right] {\footnotesize$O(P_{3})$};
		\node  at  (3,4) {\footnotesize$\Omega$};
		\node  at  (3,1.5) {\footnotesize$\Gamma_{sym}$};
		\node  at  (4.7,4) {\footnotesize$\Gamma_{wing}$};
		\node  at  (1.5,5) {\footnotesize$\Gamma_{cone}^{\infty}$};
		\draw (2.78,5.45)--(4,7.6) node  [right]  {\footnotesize$P_{5}$};
		\fill (2,2)circle(1.3pt);
		\fill (2.72,5.33)circle(1.3pt);
		\fill (4, 6.1)circle(1.3pt);
		\fill (4,7)circle(1.3pt);
		\fill (4,7.6)circle(1.3pt);
		\end{tikzpicture}
		\label{fg16}
	}\qquad
	\subfigure[$\sigma=\sigma_{0}$.]{
		\begin{tikzpicture}[smooth, scale=0.7]
		\draw  [-latex](7,2)--(11.8,2)node[right]{\footnotesize$\xi_{1}$};
		\draw  [-latex](10,1)--(10,8)node[above]{\footnotesize$\xi_{2}$};
		\draw  (8,2)  node [below ]{\footnotesize$P_{2}$} to [out=90, in =225] (10,7) node [right]{\footnotesize$P_{0}$} ;
		\node  at (10,2)[below right ]{\footnotesize$O(P_{3})$};
		\node  at (9,4){\footnotesize${U}$};
		\node  at (9,1.5){\footnotesize$\Gamma_{sym}$};
		\node  at (10.7,4){\footnotesize$\Gamma_{wing}$};
		\node  at (7.5,5){\footnotesize$\Gamma_{cone}^{\infty}$};
		\fill (10,7)circle(1.3pt);
		\fill (8,2)circle(1.3pt);
		\end{tikzpicture}
		\label{fg4}
	}
	\label{fg416}
	\caption{Domain for the case of shocks.}
\end{figure}
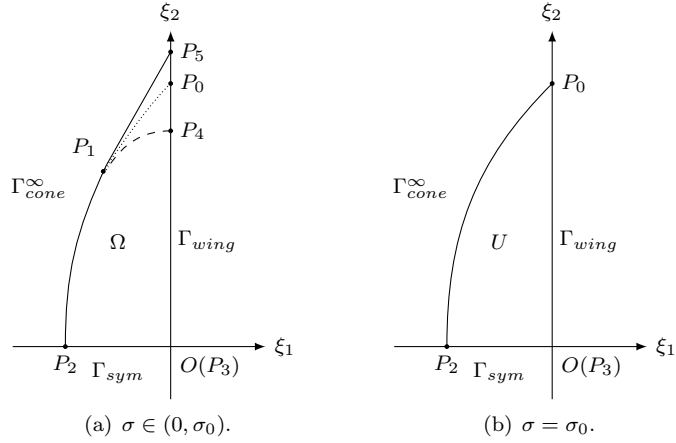

Let us denote by $\Gamma_{cone}^{\infty}$ and $\Gamma_{cone}^{\sigma}$ the arcs $P_{1}P_{2}$ and $P_{1}P_{4}$, respectively; denote by $\Gamma_{sym}$ and $\Gamma_{wing}$ the lines $P_{2}P_{3}$ and $P_{3}P_{5}$, respectively. In addition, let $U$ be the domain $P_{1}P_{2}P_{3}P_{5}$, and $\Omega$ the domain $P_{1}P_{2}P_{3}P_{4}$.

Finally, we conclude \Cref{sec:2.1,sec:2.2} by the following proposition.

\begin{proposition}\label{Pro1}
	Assume that the state $(\rho_{\infty},q_{\infty})$ of the oncoming flow is uniform and supersonic, and the wing $W_{\sigma}$ is a triangular plate given by \eqref{eq:1.2}. Then we can find a critical angle $\alpha_{0}=\alpha_{0}(\rho_{\infty},q_{\infty})\in(0,{\pi}/{2})$ so that for any fixed $\alpha\in(0,\alpha_{0})$, there exists $\sigma_{0}=\sigma_{0}(\rho_{\infty},q_{\infty},\alpha)\in(0,{\pi}/{2})$ such that
	\begin{enumerate}
		\item [i)] for $\sigma\in(0,\sigma_{0})$, there is a uniform flow in the domain $U\setminus\Omega$ $($see $\Cref{fg16})$, and the corresponding potential function satisfies
		\begin{equation*}
			\psi=\frac{v_{2\sigma}\xi_2+v_{3\sigma}}{\sqrt{B_{\infty}}}.
		\end{equation*}
		\item [ii)] for $\sigma=\sigma_{0}$, the boundary $\Gamma_{cone}^{\sigma}$ degenerates into the point $P_{0}$, and the domain $\Omega$ coincides with $U$ $($see $\Cref{fg4})$. In particular, there is no uniform flow behind the shock.
	\end{enumerate}
	Here $\alpha_0$ and $\sigma_{0}$ are defined by \eqref{eq:2.6} and \eqref{eq:2.7}, respectively.
\end{proposition}

\subsection{BVP for a nonlinear mixed-type equation}\label{sec:2.3}
In this subsection, we will reformulate \cref{prob1} in the conical coordinates. Before proceeding further, we introduce a useful notation
\begin{equation*}
	D^{2}f[\bm{a},\bm{b}]:= \sum^{2}_{i,j=1} a_{i}b_{j}\partial_{ij}f\quad\text{for $f\in C^2$ and $\bm{a},\bm{b}\in\mathbb{R}^{2}$}.
\end{equation*}

Let us first derive the potential equation for $\psi$. By \eqref{eq:1.3} and \eqref{eq:2.32}--\eqref{eq1:2.12}, we obtain
\begin{equation*}\label{eq:2.39}
	\mathrm{div}(\rho(D\psi-(\psi-D\psi\cdot\bm{\xi})\bm{\xi}))+2\rho(\psi-D\psi\cdot\bm{\xi})=0,
\end{equation*}
or equivalently,
\begin{equation}\label{eq:2.41}
	a^{2}(\Delta\psi+D^{2}\psi[\bm{\xi},\bm{\xi}])-D^{2}\psi[D\psi-\chi\bm{\xi},D\psi-\chi\bm{\xi}]=0,
\end{equation}
where $\Delta$ and $D$ denote the Laplacian and the gradient operator with respect to $\bm{\xi}$, respectively; $\chi$ is given by
\begin{equation}\label{eq:2.43}
	\chi=\psi-D\psi\cdot\bm{\xi}.
\end{equation}
Define
\begin{equation}\label{eq:2.40}
	L^{2}:=\frac{|D\psi|^{2}+|\psi-D\psi\cdot\bm{\xi}|^{2}-\frac{\psi^{2}}{1+|\bm{\xi}|^{2}}}{a^{2}}.
\end{equation}
The type of equation \eqref{eq:2.41} is determined by $L$. To be specific, equation \eqref{eq:2.41} is hyperbolic if $L>1$, elliptic if $L<1$, and parabolic degenerate if $L=1$. It follows from \eqref{eq:2.33} and \eqref{eq:2.40} that equation \eqref{eq:2.41} is hyperbolic in the domain $U\setminus\Omega$ and parabolic degenerate on the arc $\Gamma_{cone}^{\infty}\cup\Gamma_{cone}^{\sigma}$.

Then we simplify the form of the boundary conditions. Since we have shown in \Cref{lemma2} that both the function $\psi_\infty$ on $\mathcal{C}_{\infty}$ and the function $\psi_\sigma$ on $\mathcal{C}_{\sigma}$ are positive, it follows from \eqref{eq:2.34} and \eqref{eq:2.40} that $L=1$ is equivalent to
\begin{equation}\label{eq:3.1}
	\psi=\sqrt{1+|\bm{\xi}|^{2}}.
\end{equation}
In other words, \cref{prob1} can be rewritten as

\begin{problem}\label{prob2}
	Let $\alpha_{0}$, $\alpha$ and $\sigma_{0}$ be as in $\Cref{Pro1}$. Then, for any $\sigma\in(0,\sigma_{0}]$, we expect to seek a solution $\psi$ of the following boundary value problem:
	\begin{equation}\label{eq:2.42}
		\begin{cases}
			\text{Equation}~\eqref{eq:2.41}\quad&\text{in $\Omega$},\\
			\psi=\sqrt{1+|\bm{\xi}|^{2}}\quad&\text{on $\Gamma_{cone}^{\infty}\cup\Gamma_{cone}^{\sigma}$},\\
			D\psi\cdot\bm{\nu}_{w}=0\quad&\text{on $\Gamma_{wing}$},\\
			D\psi\cdot\bm{\nu}_{sy}=0\quad&\text{on $\Gamma_{sym}$},
		\end{cases}
	\end{equation}
	where $\bm{\nu}_{w}=(1,0)$ and $\bm{\nu}_{sy}=(0,-1)$ are the exterior normals to $\Gamma_{wing}$ and $\Gamma_{sym}$, respectively. We emphasize here that when $\sigma=\sigma_0$, the domain $\Omega$ coincides with $U$,  and the boundary $\Gamma_{cone}^{\sigma}$ degenerates into the point $P_0$.
\end{problem}

Correspondingly, to prove \Cref{thm: Main}, we only need to show that

\begin{theorem}\label{thm4}
	Let $\alpha_{0}$, $\alpha$ and $\sigma_{0}$ be as in $\Cref{Pro1}$. Then, for any $\sigma\in(0,\sigma_{0}]$, $\cref{prob2}$ admits a unique solution $\psi$ satisfying
	\begin{equation*}
		\psi\in C^{\infty}(\bar{\Omega}\setminus\overline{\Gamma_{cone}^{\infty}\cup\Gamma_{cone}^{\sigma}})\cap Lip(\bar{\Omega})
	\end{equation*}
	and
	\begin{equation*}
		\psi >\sqrt{1+|\bm{\xi}|^{2}} \quad\text{in}~ \bar{\Omega}\setminus\overline{\Gamma_{cone}^{\infty}\cup\Gamma_{cone}^{\sigma}}.
	\end{equation*} 	
\end{theorem}

Before proving \Cref{thm4}, we briefly make some comments on equation \eqref{eq:2.41}. This equation is hyperbolic in the domain $U\setminus\Omega$ and degenerate on the boundary $\Gamma_{cone}^{\infty}\cup\Gamma_{cone}^{\sigma}$. Then, it is natural to find a solution such that equation \eqref{eq:2.41} is elliptic in $\bar{\Omega}\setminus\overline{\Gamma_{cone}^{\infty}\cup\Gamma_{cone}^{\sigma}}$ (for example, see \cite{Serre11}). We also expect to know whether there are some parabolic bubbles inside the domain. For this purpose, we only assume $\psi\geq \sqrt{1+|\bm{\xi}|^{2}}$ in $\bar{\Omega}\setminus\overline{\Gamma_{cone}^{\infty}\cup\Gamma_{cone}^{\sigma}}$ afterwards. Fortunately, with the help of the auxiliary function given below, we can prove that no such parabolic bubbles exist in this domain. The detailed proof of \Cref{thm4} will be given in \Cref{sec:3}.

\section{Unique solvability of \cref{prob1}}\label{sec:3}
Notice that equation \eqref{eq:2.41} is of reflection symmetry with respect to the axes. In addition, the boundaries $\Gamma_{cone}^{\infty}$ and $\Gamma_{cone}^{\sigma}$ are perpendicular to the $\xi_{1}$-axis and the $\xi_{2}$-axis, respectively. Therefore, we can reflect the domain $\Omega$ with respect to the axes to obtain a domain $\Omega_{ext}$, and accordingly extend the function $\psi$ to the domain $\Omega_{ext}$ by an even reflection. For simplicity of notation, we still write $\psi$ for the function after extension. It should be noted that, when $\sigma=\sigma_{0}$ (see \Cref{fg4}), the shock $\Gamma_{cone}^{\infty}$ is not perpendicular to the $\xi_{2}$-axis, and thus the corner point $P_{0}$ cannot be removed by reflecting the domain ${U}$ about $\xi_{2}$-axis. In what follows, we mainly focus on the case $\sigma\in(0,\sigma_{0})$, and only give a brief explanation for the case $\sigma=\sigma_{0}$, if necessary.

\subsection{Comparison principle}\label{sec:3.2.2}
After the extension above, the mixed boundary value problem \eqref{eq:2.42} is reduced to a Dirichlet problem as follows
\begin{equation}\label{eq1:3.2}
	\begin{cases}
		\text{Equation}~\eqref{eq:2.41}\quad&\text{in $\Omega_{ext}$},\\
		\psi=\sqrt{1+|\bm{\xi}|^{2}}\quad&\text{on $\partial\Omega_{ext}$}.
	\end{cases}
\end{equation}

Now, our purpose is to seek a solution to problem \eqref{eq1:3.2} with $\psi\geq \sqrt{1+|\bm{\xi}|^{2}}$ in $\Omega_{ext}$. The main difficulty in solving this problem is that the type of equation \eqref{eq:2.41} in $\Omega_{ext}$ cannot be determined in advance. Thus, we expect to find a sufficient condition to ensure the uniform ellipticity of \eqref{eq:2.41}. It can be verified that equation \eqref{eq:2.41} is uniformly elliptic if and only if there exists a positive number $\varepsilon_{0}>0$ such that  $L^2<1-\varepsilon_{0}$, where $L^2$ is defined by \eqref{eq:2.40}.

Set
\begin{equation}\label{eq:3.3}
	{w}:=\frac{\psi}{\sqrt{1+|\bm{\xi}|^{2}}}.
\end{equation}
A simple computation gives
\begin{align}
	D\psi=\frac{1}{\sqrt{1+|\bm{\xi}|^{2}}}\Big(w\bm{\xi}+(1+|\bm{\xi}|^{2})D{w}\Big),\label{eq3:3.15}\\
	a^2=w^2-1+(1+|\bm{\xi}|^{2})(|D{w}|^2+|D{w}\cdot\bm{\xi}|^2).\label{eq1:3.15}
\end{align}
Since the domain $\Omega_{ext}$ is bounded, we substitute \eqref{eq:2.40} and \eqref{eq3:3.15}--\eqref{eq1:3.15} into the relation $L^2<1-\varepsilon_{0}$ to obtain
\begin{equation}\label{eq1:3.14}
	\frac{|D{w}|^2+|D{w}\cdot\bm{\xi}|^2}{w^2-1}<C(\varepsilon_{0}),
\end{equation}
where the constant $C(\varepsilon_{0})$ depends only on $\varepsilon_{0}$. It follows from  \eqref{eq:3.3} and \eqref{eq1:3.14} that if there exist a bounded constant $C$ and a positive number $\varepsilon_{0}>0$ such that $\psi$ satisfies
\begin{align}
	\sqrt{1+|\bm{\xi}|^{2}}+\varepsilon_{0}\leq\psi&<C,\label{eq4}\\
	\vert D\psi\vert&<C,\label{eq5}
\end{align}
then equation \eqref{eq:2.41}  is uniformly elliptic.

To establish estimates \eqref{eq4}--\eqref{eq5}, as well as the uniqueness of the solutions of problem \eqref{eq1:3.2}, we may expect that equation \eqref{eq:2.41} satisfies a comparison principle, but this is not easy to verify directly. Fortunately, we can proceed by employing the auxiliary function $w$. By \eqref{eq:2.41} and \eqref{eq:3.3}, the equation for $w$ has the form
\begin{multline}\label{eq:3.4}
	a^{2}(\Delta{w}+D^{2}{w}[\bm{\xi},\bm{\xi}])-(1+|\bm{\xi}|^{2})D^{2}{w}[D{w}+(D{w}\cdot\bm{\xi})\bm{\xi},D{w}+(D{w}\cdot\bm{\xi})\bm{\xi}]\\
	+2(w^2-1)D{w}\cdot\bm{\xi}+\frac{w(a^2+w^2-1)}{1+|\bm{\xi}|^{2}}=0.
\end{multline}
For notational simplicity, we denote by $\mathcal{N}_1 w$ the left-hand side of \eqref{eq:3.4}. Now, we prove that equation \eqref{eq:3.4}

\begin{lemma}\label{lemma3}
Let $\Omega_{D}\subset \mathbb{R}^{2}$ be an open bounded domain. Also, let ${w}_{\pm}\in C^{0}(\overline{\Omega_{D}})\cap C^2(\Omega_{D})$ satisfy ${w}_{\pm}>1$ in $\Omega_{D}$. Assume that the operator $\mathcal{N}_1$ is locally uniformly elliptic with respect to either ${w}_{+}$ or $w_-$, and there hold
	\begin{equation*}
		\mathcal{N}_1 w_-\geq0, \quad\mathcal{N}_1 w_+\leq0\quad\text{in}~\Omega_{D}
	\end{equation*}
with $w_-\leq w_+$ on $\partial \Omega_{D}$. Then, it follows that ${w}_{-}\leq{w}_{+}$ in $\Omega_{D}$.
\end{lemma}

\begin{proof}
	Owing to ${w}_{\pm}> 1$ in $\Omega_{D}$, we introduce a function $z>0$ defined implicitly by
	\begin{equation}\label{eq:3.6}
		{w}(\bm{\xi})=\cosh z(\bm{\xi}),
	\end{equation}
	as demonstrated in \cite{Serre11}. Since
	\begin{align}
		Dw&=\sinh zDz,\quad Dw\cdot\bm{\xi}=\sinh zDz\cdot\bm{\xi},\\
		a^2&=\sinh^2 z\big(1+(1+|\bm{\xi}|^2)(|Dz|^2+|Dz\cdot\bm{\xi}|^2)\big),\label{eq:3.7}
	\end{align}
	it follows from \eqref{eq:3.4} that
	\begin{multline}\label{eq:3.9}
		\frac{a^2}{\sinh^2 z}(\Delta z+D^{2}z[\bm{\xi},\bm{\xi}])-(1+|\bm{\xi}|^{2})D^{2}z[Dz+(Dz\cdot\bm{\xi})\bm{\xi},Dz+(Dz\cdot\bm{\xi})\bm{\xi}]\\
		+2\Big(Dz\cdot\bm{\xi}+\frac{a^2}{(1+|\bm{\xi}|^{2})\sinh^2 z\tanh z}\Big)=0.
	\end{multline}
	
	We may as well assume that the operator $\mathcal{N}_1$ is locally uniformly elliptic with respect to $w_+$. Then, we see from \eqref{eq1:3.14} and \eqref{eq:3.6} that the operator defined by the left-hand side of \eqref{eq:3.9} is also locally uniformly elliptic with respect to the corresponding ${z}_{+}$.  In addition, we notice from \eqref{eq:3.7} that the function ${a^2}/{\sinh^2 z}$ is independent of $z$, which means that the principal coefficients of \eqref{eq:3.9} depend only on $Dz$ and $\bm{\xi}$. Also, the lower-order term of  \eqref{eq:3.9} is non-increasing in $z$ for each $(\bm{\xi},Dz)\in \Omega_{D}\times \mathbb{R}^{2}$.
	
	By Theorem 10.1 in \cite{GT01}, such an equation satisfies the comparison principle, in the sense that if $z_{+}$ is a super-solution and $z_{-}$ is a sub-solution of equation \eqref{eq:3.9}, with $z_{-}\leq z_{+}$ on the boundary $\partial\Omega_{D}$, then there holds $z_{-}\leq z_{+}$ everywhere in the domain $\Omega_{D}$. Since the function $\cosh{z}$ is increasing in $(0,+\infty)$, equation \eqref{eq:3.4} also satisfies the comparison principle. This completes the proof.
\end{proof}

\begin{remark}\label{remark:1}
	From \eqref{eq:3.3} and the boundedness of $\Omega_{ext}$, we know that equation \eqref{eq:3.4} has the same type as equation \eqref{eq:2.41}, and solutions of these two problems have the same regularity. So, under the assumption that $\psi\geq\sqrt{1+|\bm{\xi}|^{2}}$ in $\Omega_{ext}$, if we can prove $\psi>\sqrt{1+|\bm{\xi}|^{2}}$ in $\Omega_{ext}$ and equation \eqref{eq:2.41} is locally uniformly elliptic in $\Omega_{ext}$, then the uniqueness of the solution in $C^0(\overline{\Omega_{ext}})\cap C^{2}(\Omega_{ext})$ to problem \eqref{eq1:3.2} can be obtained from \Cref{lemma3} immediately.
\end{remark}

\subsection{Strategy of the proof}\label{sec:3.2.1} In this subsection, we will give a strategy of the proof for the existence of the solution to problem \eqref{eq1:3.2}. Inspired by the work in \cite{Serre09,Serre11} and estimate \eqref{eq4}, we consider the following Dirichlet problem:
\begin{equation}\label{eq:3.13}
	\mathcal{F}(\mu,\psi):=a^{2}(\Delta\psi+D^{2}\psi[\bm{\xi},\bm{\xi}])-\mu D^{2}\psi[D\psi-\chi\bm{\xi},D\psi-\chi\bm{\xi}]=0\quad\text{in $\Omega_{ext}$}
\end{equation}
with
\begin{equation}\label{eq:3.14}
	\psi=\sqrt{1+\vert\bm{\xi}\vert^2}+\varepsilon \quad\text{on $\partial\Omega_{ext}$},
\end{equation}
where $\mu\in[0,1]$ and $\varepsilon>0$ are parameters, and $\chi$ is given by \eqref{eq:2.43}. Obviously, equation \eqref{eq:3.13} can be rearranged in the form
\begin{equation}\label{eq1:3.16}
	\bm{A}(\mu;\bm{\xi}, \psi,D\psi):D^{2}\psi :=\sum^{2}_{i,j=1} A_{ij}\partial_{ij}\psi=0,
\end{equation}
where $\bm{A}(\mu;\bm{\xi}, \psi,D\psi)$ has its expression as shown in \eqref{eq:3.13}.

We find that when $\mu=0$, equation \eqref{eq:3.13} is reduced to a  linear elliptic equation. Such a property motivates us to solve the problem \eqref{eq1:3.2} by the method of continuity. To apply this approach, we analyze equation \eqref{eq:3.13} in the same way as equation \eqref{eq:2.41} in \Cref{sec:3.2.2}.

Denote by $\psi_{\mu,\varepsilon}$ a solution to problem \eqref{eq:3.13}--\eqref{eq:3.14},  and define
\begin{equation*}
	L^{2}_\mu:=\frac{\mu\big(|D\psi|^{2}+|\psi-D\psi\cdot\bm{\xi}|^{2}-\frac{\psi^{2}}{1+|\bm{\xi}|^{2}}\big)}{a^{2}}.
\end{equation*}
It can be verified that equation \eqref{eq:3.13} is elliptic if $L^{2}_\mu<1$, i.e.,
\begin{equation*}
	\psi_{\mu,\varepsilon}>\sqrt{(1+|\bm{\xi}|^2)\Big(1+\frac{\mu-1}{\mu}a^2\Big)}\quad\text{for } \mu\in(0,1].
\end{equation*}
We deduce from this relation that if $\psi_{\mu,\varepsilon}>\sqrt{1+|\bm{\xi}|^2}$, then equation \eqref{eq:3.13} is always elliptic for any $\mu\in[0,1]$. Moreover, using the analysis as in \Cref{sec:3.2.2}, we know that for any $\mu\in[0,1]$, equation \eqref{eq:3.13} is uniformly elliptic if we can find a positive number $\varepsilon_{0}>0$ and a bounded constant $C$ so that
\begin{align}
	\sqrt{1+|\bm{\xi}|^2}+\varepsilon_{0}\leq  \psi_{\mu,\varepsilon}&<C,\label{eq:3.15}\\
	\vert D\psi_{\mu,\varepsilon}\vert&<C.\label{eq2:3.15}
\end{align}

Next, let us verify that the corresponding equation for $w_{\mu,\varepsilon}$ also satisfies the comparison principle (i.e., \Cref{lemma3}), which will enable us to obtain estimate \eqref{eq:3.15} later. By \eqref{eq:3.3} and \eqref{eq:3.13}, the equation for $w_{\mu,\varepsilon}$ is
\begin{multline}\label{eq:3.20}
	a^{2}(\Delta{w}+D^{2}{w}[\bm{\xi},\bm{\xi}])-\mu(1+|\bm{\xi}|^{2})D^{2}{w}[D{w}+(D{w}\cdot\bm{\xi})\bm{\xi},D{w}+(D{w}\cdot\bm{\xi})\bm{\xi}]\\
	+2\big((1-\mu)a^{2}+\mu(w^2-1)\big)D{w}\cdot\bm{\xi}\\+\Big((2-\mu)a^2+\mu(w^2-1)\Big)\frac{w}{1+|\bm{\xi}|^{2}}=0.
\end{multline}
For simplicity, denote by $\mathcal{N}_\mu w$ the left-hand side of equation \eqref{eq:3.20}. As before, we also utilize the auxiliary function $z_{\mu,\varepsilon}$. Using \eqref{eq:3.6}, we have
\begin{multline}\label{eq:3.16}
	(1+m(\bm{\xi},Dz))(\Delta z+D^{2}z[\bm{\xi},\bm{\xi}])-\mu(1+|\bm{\xi}|^{2})D^{2}z[Dz+(Dz\cdot\bm{\xi})\bm{\xi},Dz+(Dz\cdot\bm{\xi})\bm{\xi}]\\
	+2\Big(1+(1-\mu)m(\bm{\xi},Dz)\Big)Dz\cdot\bm{\xi}\\+\Big(2+(1-\mu)m(\bm{\xi},Dz)\Big)\frac{1+m(\bm{\xi},Dz)}{(1+|\bm{\xi}|^{2})\tanh z}=0,
\end{multline}
where $m(\bm{\xi},Dz)$ is defined as
\begin{equation*}
	m(\bm{\xi},Dz):=(1+|\bm{\xi}|^2)(|Dz|^2+|Dz\cdot\bm{\xi}|^2).
\end{equation*}
Let us write equation \eqref{eq:3.16} in the form
\begin{equation}\label{eq1:3.1}
	\sum^{2}_{i,j=1} a_{ij}(\mu;\bm{\xi},Dz)\partial_{ij}z+H(\mu;\bm{\xi},z,Dz)=0.
\end{equation}
It follows from \eqref{eq:3.16} that for any $\mu\in[0,1]$, the coefficients $a_{ij}$ are independent of $z$, and the lower-order term $H(\mu;\bm{\xi},z,Dz)$ is non-increasing in $z$ for each $(\bm{\xi},Dz)\in \Omega_{ext}\times \mathbb{R}^{2}$. Then, we see that equation \eqref{eq:3.16} has the same form as equation \eqref{eq:3.9}. Consequently, the conclusion in \Cref{lemma3} is still valid for equation \eqref{eq:3.20}.

According to the above analysis, we now seek a solution to problem \eqref{eq1:3.2} by the continuity method. Define
\begin{equation}\label{eq1:111}
	\begin{aligned}
		J_{\varepsilon}:=\{\mu\in[0,1]:~&\text{such that}~\psi_{\mu,\varepsilon}\in C^0(\overline{\Omega_{ext}})\cap C^{2}(\Omega_{ext})~\text{satisfies}\\& \eqref{eq:3.13}\text{--}\eqref{eq:3.14} ~\text{with}~\psi_{\mu,\varepsilon}\geq \sqrt{1+\vert\bm{\xi}\vert^2}+\varepsilon~\text{in}~\Omega_{ext}\}.
	\end{aligned}
\end{equation}
Then the strategy of our proof can be described as follows. For any fixed $\varepsilon>0$, we first prove $J_{\varepsilon}=[0,1]$. It is required to verify that $J_{\varepsilon}$ is open, closed and not empty. To this end, we will establish a prior estimates which are independent of $\mu$ and $\varepsilon$ in \Cref{sec:3.2.6}. Next, we show that the limit of $\psi_{1,\varepsilon}$ as $\varepsilon\rightarrow 0+$ is exactly the solution to problem \eqref{eq1:3.2}.

\subsection{Lipschitz estimate}\label{sec:3.2.6}
For simplicity of notation, the subscripts of $\psi_{\mu,\varepsilon}$, $w_{\mu,\varepsilon}$ and $\Phi_{\mu,\varepsilon}$ will be omitted in this subsection.

Let us first derive estimate \eqref{eq:3.15}. Since equation \eqref{eq:3.20} satisfies the comparison principle, it is more convenient to derive the $L^\infty$-estimate for $w$, which in turn leads to the corresponding estimate for $\psi$ directly.

\begin{lemma}\label{lemma: inf-estimate}
	Let $w \in C^{0}(\overline{\Omega_{ext}})\cap C^{2}(\Omega_{ext})$ satisfy $\mathcal{N}_\mu w=0$, ${w}>1$ in $\Omega_{ext}$ and $w=1+\varepsilon$ on $\partial \Omega_{ext}$. Then there exists a constant $C>0$, independent of $\mu$ and $\varepsilon$, such that
	\begin{equation}\label{eq1:3.13}
		1+\varepsilon<{w}\leq C\quad\text{in $\Omega_{ext}$},
	\end{equation}
	where $\mathcal{N}_\mu w=0$ denotes equation \eqref{eq:3.20}.
\end{lemma}

\begin{proof}
	The main work of the proof is to find proper sub- and super-solutions to the equation $\mathcal{N}_\mu w=0$ with boundary condition $w|_{\partial\Omega_{ext}}=1+\varepsilon$. For this purpose, we first seek an exact solution to the equation $\mathcal{N}_\mu w=0$. By \eqref{eq:2.32}--\eqref{eq1:2.12} and \eqref{eq:3.13}, the corresponding $\Phi$ is a solution of
	\begin{equation}\label{eq:3.25}
		(|\nabla_{\bm{x}}\Phi|^2-B_{\infty})\Delta_{\bm{x}}\Phi-\mu D^2_{\bm{x}}\Phi[\nabla_{\bm{x}}\Phi,\nabla_{\bm{x}}\Phi]=0.
	\end{equation}
	Notice that equation \eqref{eq:3.25} is independent of $\Phi$ itself. Then,  for any given constant vector $\bm{\eta}\in \mathbb{R}^{3}$, the linear function $\Phi^{\bm{\eta}}=\sqrt{B_{\infty}}\bm{\eta}\cdot\bm{x}$ is a solution to \eqref{eq:3.25}. From \eqref{eq:2.32}--\eqref{eq1:2.12} and \eqref{eq:3.3}, we know that
	\begin{equation}\label{eq:3.26}
		{w}^{\bm{\eta}}(\bm{\xi})=\frac{\bm{\eta}\cdot(\bm{\xi},1)}{\sqrt{1+|\bm{\xi}|^{2}}}
	\end{equation}
	is an exact solution to the equation $\mathcal{N}_\mu w=0$. A simple computation gives
	\begin{equation}\label{eq1:3.26}
		\partial_i {w}^{\bm{\eta}}(\bm{\xi})=\frac{\sqrt{1+|\bm{\xi}|^{2}}\eta_i-{w}^{\bm{\eta}}\xi_i}{1+|\bm{\xi}|^{2}}\quad\text{ for }i=1,2,
	\end{equation}
	where $\partial_i$ stands for $\partial_{\xi_i}$.
	
	For any fixed $\varepsilon>0$, let us define the sets
	\begin{align}
		\Sigma_{+}&:=\{\bm{\eta}\in\mathbb{R}^{3}: {w}^{\bm{\eta}}>1+\varepsilon\text{ on }\partial\Omega_{ext}\},\\
		\Sigma_{-}&:=\{\bm{\eta}\in\mathbb{R}^{3}: {w}^{\bm{\eta}}<1+\varepsilon\text{ on }\partial\Omega_{ext}\}.\label{eq:3.2}
	\end{align}
	It follows from \eqref{eq:3.26} that for any fixed constant vector $\bm{\eta}\in \mathbb{R}^{3}$, ${w}^{\bm{\eta}}(\bm{\xi})$ is a decreasing function of the angle between $\frac{(\xi_{1},\xi_{2},1)}{\sqrt{1+|\bm{\xi}|^{2}}}$ and $\bm{\eta}$, which implies that when $\bm{\eta}\in \Sigma_{+}$, the function ${w}^{\bm{\eta}}$ is larger than $1+\varepsilon$ everywhere in $\Omega_{ext}$. Furthermore, for any compact subdomain $\Omega^c_{sub}\subset\Omega_{ext}$, we have
	\begin{equation}\label{eq1:3.20}
		{w}^{\bm{\eta}}>1+\varepsilon+\delta\quad\text{in $\Omega^c_{sub}$},
	\end{equation}
	where the constant $\delta>0$ depends only on $\Omega^c_{sub}$.
	
	We deduce from \eqref{eq:3.26}--\eqref{eq1:3.26} and \eqref{eq1:3.20} that when $\bm{\eta}\in \Sigma_{+}$, the operator $\mathcal{N}_\mu$ is locally uniformly elliptic with respect to ${w}^{\bm{\eta}}$. Then, from \Cref{lemma3}, it follows that $w\leq {w}^{\bm{\eta}} $ in $\Omega_{ext}$. Taking ${w}^{+}$ as the infimum of all these ${w}^{\bm{\eta}}$, namely,
	\begin{equation}\label{eq:3.28}
		{w}^{+}=\inf\{{w}^{\bm{\eta}}:\bm{\eta}\in \Sigma_{+}\},
	\end{equation}
	we obtain $w\leq {w}^{+}$ in $\Omega_{ext}$ and $w^+\geq 1+\varepsilon$ on $\partial\Omega_{ext}$.  In particular, the convexity of the domain $\Omega_{ext}$ yields
	\begin{equation}\label{eq:3.29}
		{w}^{+}=1+\varepsilon\quad\text{on $\partial\Omega_{ext}$}.
	\end{equation}
	
	The situation is analogous for sub-solutions. Define
	\begin{equation}\label{eq:3.30}
		{w}^{-}:=\sup\{{w}^{\bm{\eta}}:\bm{\eta}\in \Sigma_{-}\}.
	\end{equation}
	Then we have $\mathcal{N}_{\mu}w^-\geq0$ in $\Omega_{ext}$.
	
	We claim here that for any compact subdomain $\Omega^c_{sub}\subset\Omega_{ext}$, there exists a positive number $\delta>0$ depending only on $\Omega^c_{sub}$ such that
	\begin{equation}\label{eq1:3.21}
		{w}^{-}\geq 1+\varepsilon+\delta\quad\text{in $\Omega^c_{sub}$}.
	\end{equation}
	For any fixed point $\bm{\xi}_{0}=(\xi_{10},\xi_{20})\in \Omega^c_{sub}$, we can choose a suitably small constant $0<\delta\ll 1$ and a constant vector $\bm{\eta}_{0}=(1+\varepsilon+\delta)\frac{(\xi_{10},\xi_{20},1)}{\sqrt{1+|\bm{\xi}_{0}|^{2}}}$ such that $\bm{\eta}_{0}$ belongs to $\Sigma_{-}$. Therefore, by \eqref{eq:3.30} and the definition of $w^-$, we obtain the relation \eqref{eq1:3.21}.
	
	By the continuity of the function ${w}^{-}$, \eqref{eq1:3.21} implies ${w}^{-}>1+\varepsilon$ in $\Omega_{ext}$ and ${w}^{-}\geq 1+\varepsilon$ on $\partial\Omega_{ext}$. In addition, it is clear that ${w}^{-}\leq 1+\varepsilon$ on $\partial \Omega_{ext}$. Consequently, we have
	\begin{equation}\label{eq:3.31}
		{w}^{-}= 1+\varepsilon\quad\text{on $\partial\Omega_{ext}$}.
	\end{equation}
	It follows from \eqref{eq:3.26}--\eqref{eq1:3.26} and \eqref{eq1:3.21} that the operator $\mathcal{N}_\mu$ is locally uniformly elliptic with respect to ${w}^{-}$. According to \Cref{lemma3}, we obtain $w\geq w^-$ in $\Omega_{ext}$.
	
	We conclude from the discussion above that
	\begin{equation}\label{eq1:3.10}
		1+\varepsilon<{w}^{-}\leq{w}\leq{w}^{+}\quad\text{in $\Omega_{ext}$},
	\end{equation}
	where ${w}^{\pm}$ satisfy the conditions \eqref{eq:3.29} and \eqref{eq:3.31} on $\partial\Omega_{ext}$. This completes the proof.
\end{proof}

\begin{remark}\label{remmarklbs6}
	It is shown from the above proof that only quality \eqref{eq:3.29} depends on the convexity of $\Omega_{ext}$, while \eqref{eq:3.31} holds even for the domain $\Omega_{ext}$ being non-convex.
\end{remark}

On account of the boundedness of $\Omega_{ext}$, we do not distinguish the small numbers $\varepsilon$ and $\sqrt{1+|\bm{\xi}|^2} \varepsilon$,  and write them as $\varepsilon$ all the time. Then from \eqref{eq:3.3} and \eqref{eq1:3.10}, we obtain
\begin{equation}\label{eq1:3.17}
	\sqrt{1+|\bm{\xi}|^2}+\varepsilon<\psi^-\leq{\psi}\leq\psi^+\leq C\quad\text{in $\Omega_{ext}$},
\end{equation}
where $\psi^{\pm}=\sqrt{1+|\bm{\xi}|^2} w^{\pm}$, and the constant $C$ is independent of $\mu$ and $\varepsilon$. Moreover, there still holds
\begin{equation}\label{eq1:3.18}
	{\psi}={\psi}^{\pm}\quad\text{on $\partial\Omega_{ext}$}.
\end{equation}

Next, we establish the Lipschitz estimate \eqref{eq2:3.15}.

\begin{lemma}\label{lemma: lip-estimate}
	Let $\psi \in C^{0}(\overline{\Omega_{ext}})\cap C^{2}(\Omega_{ext})$ satisfy problem  \eqref{eq:3.13}--\eqref{eq:3.14} and $\psi>\sqrt{1+|\bm{\xi}|^2}$ in $\Omega_{ext}$. Then there exists a constant $C>0$, independent of $\mu$ and $\varepsilon$, such that
	\begin{equation}\label{eq:3.42}
		\|D{\psi}\|_{L^{\infty}(\Omega_{ext})}\leq C.
	\end{equation}
\end{lemma}

\begin{proof}
	First, we consider the estimate on the boundary. It follows from \eqref{eq1:3.17}--\eqref{eq1:3.18} that
	\begin{equation}\label{eq:3.34}
		\|D{\psi}\|_{L^{\infty}(\partial\Omega_{ext})}\leq \|D{\psi}^{+}\|_{L^{\infty}(\partial\Omega_{ext})}.
	\end{equation}
	Also, by \eqref{eq1:3.26} and \eqref{eq:3.28}, and noting $\psi^{+}=\sqrt{1+|\bm{\xi}|^2} w^{+}$, we know that the right-hand side above is bounded by a constant $C$, independent of $\mu$ and $\varepsilon$.
	
	Then, we turn to the interior Lipschitz estimate. For any function $\psi\in C^2(\Omega_{ext})$, since we can choose a sequence $\{\psi_l\}\subset C^3(\Omega_{ext})$ such that $\psi_l\rightarrow\psi$ in $C^2(\Omega_{ext})$ as $l\rightarrow \infty$, we may as well assume that $\psi\in C^3(\Omega_{ext})$.
	
	For equation \eqref{eq1:3.16}, a standard calculation shows that the function $\frac{1}{2}|D\psi|^{2}$ satisfies
	\begin{equation}\label{eq:3.37}
		\bm{A}:D^{2}(\frac{1}{2}|D\psi|^{2})=\sum^{2}_{k=1}(D\partial_k\psi)^{T}\cdot \bm{A}\cdot (D\partial_k\psi)+\sum^{2}_{i,j,k=1}A_{ij}\partial_{k}\psi\partial_{ijk}\psi.
	\end{equation}
	By differentiating $\bm{A}:D^{2}\psi=0$ with respect to $\xi_{k}$ and multiplying on its both sides by $\partial_{k}\psi$, we have
	\begin{equation}\label{eq:3.38}
		\sum^{2}_{i,j=1}\partial_{k}(A_{ij})\partial_{ij}\psi\partial_{k}\psi+\sum^{2}_{i,j=1}A_{ij}\partial_{k}\psi\partial_{ijk}\psi=0.
	\end{equation}
	Plugging \eqref{eq:3.38} into \eqref{eq:3.37} gives
	\begin{equation}\label{eq:3.39}
		\bm{A}:D^{2}(\frac{1}{2}|D\psi|^{2})=\sum^{2}_{k=1}(D\partial_k\psi)^{T}\cdot\bm{A}\cdot (D\partial_k\psi)+\sum^{2}_{k=1}h_k(\bm{\xi},\psi,D\psi)\partial_{k}(\frac{1}{2}|D\psi|^{2}),
	\end{equation}
	where $h_k$ is given by
	\begin{equation}\label{eq1:3.40}
		h_k(\bm{\xi},\psi,D\psi)=-\sum^{2}_{i,j=1}\partial_{i}(A_{kj})\frac{\partial_i \psi}{\partial_j \psi}.
	\end{equation}
	
	Also, it follows from \eqref{eq1:3.17} that equation \eqref{eq1:3.16} is elliptic in $\Omega_{ext}$. Then the matrix $\bm{A}$ is positive definite, i.e.,
	\begin{equation}\label{eq:3.40}
		\sum^{2}_{k=1}(D\partial_k\psi)^{T}\cdot \bm{A}\cdot (D\partial_k\psi)>0.
	\end{equation}
	
	We assert here that the function $|D\psi|^{2}$ cannot achieve its maximum at any interior point unless it is constant. In fact, for any interior point $\tilde{\bm{\xi}}_0\in \Omega_{ext}$, by \eqref{eq:3.39} and \eqref{eq:3.40}, this conclusion is established when $D\psi(\tilde{\bm{\xi}}_0)=0$, or both $\partial_1\psi(\tilde{\bm{\xi}}_0)$ and $\partial_2\psi(\tilde{\bm{\xi}}_0)$ equal zero. Then it suffices to discuss the case that either $\partial_1\psi(\tilde{\bm{\xi}}_0)$ or $\partial_2\psi(\tilde{\bm{\xi}}_0)$  equals zero. In either case, we can choose a rotation transformation in $(\xi_1,\xi_2)$-plane such that $\partial_1\psi\neq 0$ and $\partial_2\psi\neq 0$ at the corresponding point of $\tilde{\bm{\xi}}_0$ in the new coordinates. Also note that, since equation \eqref{eq1:3.16} and the function $|D\psi|^{2}$ are rotation invariant, the form of equation \eqref{eq:3.39} still holds under the rotation transformation. In the new coordinates, using \eqref{eq:3.39} and \eqref{eq:3.40} again, we can verify this assertion. Therefore, its maximum is achieved on the boundary, and we obtain
	\begin{equation}\label{eq1:3.11}
		\|D\psi\|_{L^{\infty}(\Omega_{ext})}\leq \|D\psi\|_{L^{\infty}(\partial\Omega_{ext})}.
	\end{equation}
	Combining this estimate and \eqref{eq:3.34}, we complete the proof of the lemma.
\end{proof}

\subsection{Continuation procedure}\label{sec:3.2.5}
Now we are ready to solve problem \eqref{eq1:3.2} by using the strategy mentioned in \Cref{sec:3.2.1}. For a fixed $\varepsilon>0$, we first prove $J_{\varepsilon}=[0,1]$, which will be divided into three steps.

\textit{Step} 1. $J_{\varepsilon}$ is not empty. To this end, we consider
\begin{equation}\label{eq:3.43}
	\mathcal{F}(0,\psi):=\Delta\psi+D^{2}\psi[\bm{\xi},\bm{\xi}]=0\quad\text{in $\Omega_{ext}$}
\end{equation}
with the boundary condition \eqref{eq:3.14}. This is a Dirichlet problem for a linear elliptic equation, which is uniquely solvable by the Fredholm alternative. Since $\Gamma_{cone}^{\infty}$ is tangent to $\Gamma_{cone}^{\sigma}$ at the point $P_{1}$ and both $\Gamma_{cone}^{\infty}$ and $\Gamma_{cone}^{\sigma}$ are smooth, the boundary $\partial\Omega_{ext}\in C^{1}$. Thanks to Theorem 6.13 and 6.17 in \cite{GT01}, we have $\psi_{0,\varepsilon}\in  C^{0}(\overline{\Omega_{ext}})\cap C^{\infty}(\Omega_{ext})$. In addition, we obtain $\psi_{0,\varepsilon}\geq \sqrt{1+|\bm{\xi}|^2}+\varepsilon$ in the $\Omega_{ext}$ by the maximum principle. Then, $0\in J_{\varepsilon}$.

\textit{Step} 2. $J_{\varepsilon}$ is closed. From \eqref{eq1:111}, we see that if $\mu\in J_{\varepsilon}$, then the corresponding function $\psi_{\mu,\varepsilon}$ satisfies estimates \eqref{eq1:3.17} and \eqref{eq:3.42}. Thus we have $\psi_{\mu,\varepsilon}\in Lip(\overline{\Omega_{ext}})\cap C^{2}(\Omega_{ext})$. Moreover, the linearized equation
\begin{equation}\label{eq:3.45}
	\bm{A}(\mu;\bm{\xi}, \psi_{\mu,\varepsilon}, D\psi_{\mu,\varepsilon}):D^2 \psi=a^{2}(\Delta \psi+D^{2} \psi[\bm{\xi},\bm{\xi}])-\mu D^{2}\psi[D\psi_{\mu,\varepsilon}-\chi\bm{\xi},D\psi_{\mu,\varepsilon}-\chi\bm{\xi}]
\end{equation}
is uniformly elliptic. By Theorem 6.17 in \cite{GT01},  we obtain the interior estimates for all the derivatives of $\psi_{\mu,\varepsilon}$, which implies $\psi_{\mu,\varepsilon}\in Lip(\overline{\Omega_{ext}})\cap  C^{\infty}(\Omega_{ext})$.

Let $(\mu_{m},\psi_{m,\varepsilon})$ be a sequence such that $\mu_{m}\in J_{\varepsilon}$ and $\psi_{m,\varepsilon}$ is the solution of the corresponding problem. Since $\psi_{m,\varepsilon}$ is bounded in $Lip(\overline{\Omega_{ext}})\cap  C^{\infty}(\Omega_{ext})$, there exists a subsequence $\mu_{m_{k}}$ such that the corresponding $\psi_{m_{k},\varepsilon}$ are convergent in $C^{0}(\overline{\Omega_{ext}})\cap C^{2}(\Omega_{ext})$. Obviously, the limit $\psi_{m_{\infty},\varepsilon}$ is a solution to problem \eqref{eq:3.13}--\eqref{eq:3.14} and $\psi_{m_{\infty},\varepsilon}\geq \sqrt{1+|\bm{\xi}|^2}+\varepsilon$ in the $\Omega_{ext}$. Then $\mu_{\infty}\in J_{\varepsilon}$, which means that $J_{\varepsilon}$ is closed.

\textit{Step} 3. $J_{\varepsilon}$ is open. Let $\mu_{0}\in J_{\varepsilon}$ and $\psi_{\mu_{0},\varepsilon}$ be the solution of the corresponding problem. We work in terms of the auxiliary function $z_{\mu,\varepsilon}$. The linearization of equation \eqref{eq1:3.1} at $z_{\mu_{0},\varepsilon}$ is
\begin{equation}\label{eq:3.46}
	\sum^{2}_{i,j=1} a_{ij}(\mu_0;\bm{\xi},Dz_{\mu_{0},\varepsilon})\partial_{ij}z+b_i(\mu_0;\bm{\xi},Dz_{\mu_{0},\varepsilon})\partial_i z+d(\mu_0;\bm{\xi},Dz_{\mu_{0},\varepsilon})z=f,
\end{equation}
where
\begin{align*}
	b_{i}&=\sum^{2}_{j=1}\partial_{p_i}a_{ij}(\mu_{0};\bm{\xi},Dz_{\mu_{0},\varepsilon})\partial_{ij}z_{\mu_{0},\varepsilon}+\partial_{p_i}H(\mu_{0};\bm{\xi},z_{\mu_{0},\varepsilon},Dz_{\mu_{0},\varepsilon})\quad\text{$i=1,2$},\\
	d&=\partial_{p_0}H(\mu_{0};\bm{\xi},z_{\mu_{0},\varepsilon},Dz_{\mu_{0},\varepsilon})
\end{align*}
with $(p_0,p_1,p_2):=(z,Dz)$. Since $d(\mu_0;\bm{\xi},Dz_{\mu_{0},\varepsilon})\leq 0$,  equation \eqref{eq:3.46} satisfies the maximum principle. Then the corresponding Dirichlet problem is uniquely solvable by the Fredholm alternative. This means that the linearized operator of the nonlinear mapping $\mu\mapsto\psi_{\mu,\varepsilon}$ is invertible at $(\mu_{0},\psi_{\mu_{0},\varepsilon})$. Therefore, by the implicit function theorem, there exists a neighborhood of $\mu_{0}$ such that the mapping $\mu\mapsto\psi_{\mu,\varepsilon}$ exists, namely, $\mu_{0}$ is an interior point of $J_{\varepsilon}$.

Note that $J_{\varepsilon}$ is both open and closed with $0\in J_{\varepsilon}$, we obtain $J_{\varepsilon}=[0,1]$, which indicates the existence of $\psi_{1,\varepsilon}$.

Next we turn to the existence of the solution to problem \eqref{eq1:3.2}.  We consider the limit of $\psi_{1,\varepsilon}$ as $\varepsilon\rightarrow 0+$.  As discussion above, we have $\psi_{1,\varepsilon}\in Lip(\overline{\Omega_{ext}})\cap C^\infty(\Omega_{ext})$. Noticing that estimates \eqref{eq1:3.17} and \eqref{eq:3.42} are independent of $\varepsilon$, then by the Arzela-Ascoli theorem, the function $\psi_{1,\varepsilon}$ is convergent in the class of $C^0(\overline{\Omega_{ext}})\cap C^2(\Omega_{ext})$.
Therefore, we can define
\begin{equation*}\label{eq:3.47}
	\psi : = \lim_{\varepsilon\rightarrow 0+}\psi_{1,\varepsilon},
\end{equation*}
which is a solution to problem \eqref{eq1:3.2}. Also note that, estimate \eqref{eq1:3.17} implies $\psi_{1,\varepsilon}>\sqrt{1+|\bm{\xi}|^2}+\varepsilon$ in $\Omega_{ext}$, though we only assume $\psi_{1,\varepsilon}\geq\sqrt{1+|\bm{\xi}|^2}+\varepsilon$ in $\Omega_{ext}$. Then,  we have $\psi>\sqrt{1+|\bm{\xi}|^2}$ in $\Omega_{ext}$; namely, there is no parabolic bubble in the domain $\bar{\Omega}\setminus\overline{\Gamma_{cone}^{\infty}\cup\Gamma_{cone}^{\sigma}}$.

By using estimates \eqref{eq1:3.17} and \eqref{eq:3.42} again, we have $\psi\in Lip(\overline{\Omega_{ext}})\cap C^2(\Omega_{ext})$. Moreover, from \eqref{eq:3.3}, \eqref{eq1:3.21} and \eqref{eq1:3.17}, it follows that for any compact subdomain $\Omega^c_{sub}\subset\Omega_{ext}$, there exists a positive number $\varepsilon_0>0$ only depending on $\Omega^c_{sub}$ such that
\begin{equation*}
	\psi \geq 1+\varepsilon_0\quad\text{in $\Omega^c_{sub}$},
\end{equation*}
which implies that $\psi$ satisfies estimates \eqref{eq4}--\eqref{eq5} in  $\Omega^c_{sub}$.  Then equation \eqref{eq:2.41}  is  locally uniformly elliptic in  $\Omega_{ext}$. Therefore the uniqueness for problem \eqref{eq1:3.2} can be obtained by \Cref{remark:1}. Furthermore,  by using Theorem 6.17 in \cite{GT01} again, we have $\psi\in Lip(\overline{\Omega_{ext}})\cap C^\infty(\Omega_{ext})$.

Since only the convexity of $\Omega_{ext}$ is required in the proof above, we can obtain the same result for the case $\sigma=\sigma_{0}$ of shocks, as well as the case of rarefaction waves. We just omit the details here. Combining \Cref{Pro1} and \Cref{thm4}, we have proved \Cref{thm: Main}.

\section{Discussion of thin delta wings}\label{sec:4}
In this section, we deal with the problem of supersonic flow over a thin delta wing with diamond cross-sections.  Let $W_{\sigma}^{\theta}$ denote such a wing, in which the angle of apex is $\pi-2\sigma$ and the angle between the upper and lower root chords is $|2\theta|$, where $\sigma\in(0,\pi/2)$ and $\theta\in(-\pi/2,0)$. As in \Cref{sec:1}, we place $W_{\sigma}^{\theta}$ in the rectangular coordinates $(x_1,x_2,x_3)$ such that it is symmetrically about the $x_{1}Ox_{3}$-plane and $x_{2}Ox_{3}$-plane, with the apex at the origin(see \Cref{fg8}), namely,
\begin{equation}\label{eq4:12}
	W_{\sigma}^{\theta}=\{(x_1,x_2,x_3):|x_{2}|<x_{3}\cot\sigma, |x_{1}|<(x_2\tan\sigma-x_3)\tan\theta,x_{3}>0\}.
\end{equation}
Obviously, when $\theta=0$ it becomes a triangle plate, which is the case discussed in the previous sections. Moreover, we assume the oncoming flow as in \Cref{sec:1}, and still use the notations as defined in \Cref{sec:1,sec:2} except the boundary $\Gamma_{wing}$, which denotes the upper surface of $W_{\sigma}^{\theta}$ in this section.

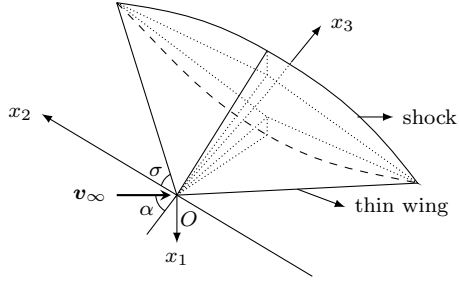
\begin{figure}[H]
	\centering
	\begin{tikzpicture}[smooth, scale=0.8]
	\draw  [-latex] (8.25,1.45)--(3.75,4.15) node [left] {\footnotesize$x_{2}$};
	\draw  (5,6)--(6,2.8)--(10,3);
	\draw  [densely dotted] (7.5,3.8)--(7.5,4.1);
	\draw  [densely dotted] (7.5,4.8)--(7.5,5.2);
	\draw  (7.5,5.2)--(6,2.8);
	\draw (5,6) to [out=-8,in=150]  (7.5,5.2) to [out=-30,in=127]  (10,3);
	\draw [dashed] (5,6) to [out=-45,in=150]  (7.5,3.8) to [out=-30,in=-192] (10,3);
	\draw  [-latex] (6,2.8)--(6,2) node [below] {\footnotesize$x_{1}$};
	\node at (6.2,2.4){\footnotesize$O$};
	\draw (6,2.8)--(5.5,2.145);
	\draw  [densely dotted]  (6,2.8)--(7.5,3.8);
	\draw  [densely dotted]  (6,2.8)--(7.5,4.5);
	\draw  [densely dotted]  (7.5,4.5)--(7.9,5.01);
	\draw  [-latex]  (7.9,5.01)--(8.4,5.625)node [right] {\footnotesize$x_{3}$};
	\draw [thick, -stealth] (5,2.8)  node[left, font=\footnotesize] {$\bm{v}_{\infty}$}  --(5.9,2.8) ;
	\draw (5.9,3.15)arc (116:165:0.3);
	\node at (5.65,3.2){\footnotesize$\sigma$};
	\draw (5.65,2.8)arc (180:240:0.3);
	\node at (5.5,2.55){\footnotesize$\alpha$};
	\draw [-latex](8,2.9)--(8.8,2.6) node[
	right] {\footnotesize{thin wing}};
	\draw [-latex](9,4.1)--(9.6,4.1) node[right] {\footnotesize{shock}};
	\draw [densely dotted] (5,6)--(7.5,4.8)--(10,3)--(7.5,4.1)--(5,6);
	\draw [densely dotted] (7.5,4.1)--(6,2.8)--(7.5,4.8);
	\end{tikzpicture}
	\caption{Geometry of a thin delta wing.}
	\label{fg8}
\end{figure}

Without loss of generality, we only discuss the case of attached shocks. Write
\begin{equation*}
	\mathcal{R}_{\sigma}^{\theta}:= \{s(x_{2},x_{3})<x_{1}<(x_3-x_2\tan\sigma)\tan\theta, x_{2}>0\},
\end{equation*}
where $x_1=s(x_{2},x_{3})$ is the equation for the shock attached to the leading edges. In view of the symmetry of the wing $W_{\sigma}^{\theta}$, it suffices to consider the problem in the region $\mathcal{R}_{\sigma}^{\theta}$. Clearly, the potential function $\Phi$ satisfies \eqref{eq:1.3}--\eqref{eq:1.4} in $\mathcal{R}_{\sigma}^{\theta}$ with boundary condition \eqref{eq1:1.2} and
\begin{equation}\label{eq1:41}
	\nabla_{\bm{x}}\Phi\cdot\bm{n}'_{w}=0 \quad\text{on $\{x_1=(x_3-x_2\tan\sigma)\tan\theta, x_2>0\}$},
\end{equation}
where $\bm{n}'_{w}=(1,\tan\theta\tan\sigma,-\tan\theta)$ is the exterior normal to the upper surface of $W_{\sigma}^{\theta}$.
By the continuity of $\Phi$, we also have \eqref{eq1:1.1} on $S_{\sigma}$. We point out here that equation \eqref{eq:1.3} is still hyperbolic in $\mathcal{R}_{\sigma}^{\theta}$, and thus there exists a Mach cone of the apex of wing.

Proceeding as in the analysis of \Cref{sec:2}, we can determine the location of the shock and the uniform flow state outside the Mach cone. The only difference is that at this point the oncoming flow  should satisfy the condition
\begin{equation}\label{eq:4.1}
	c_{\infty}<\tilde{q}_{\infty}<\frac{c_{\infty}}{\sin(\alpha_{n}-\theta_{n})}
\end{equation}
instead of \eqref{eq:2.2}, where $\theta_{n}=\arctan(\tan\theta/\cos\sigma)$ denotes the angle between the upper surface of $W_{\sigma}^{\theta}$ and the $x_{2}Ox_{3}$-plane. With the relations
\begin{equation*}
	\tilde{q}_{\infty}\sin\alpha_{n}={q}_{\infty}\sin\alpha, \quad \tilde{q}_{\infty}\cos\alpha_{n}={q}_{\infty}\cos\alpha\cos\sigma,
\end{equation*}
the right-hand side of \eqref{eq:4.1} can be reduced to
\begin{equation}\label{eq:4.2}
	q_{\infty}\sin(\alpha-\theta_{n})+q_{\infty}\cos\alpha\sin\theta_{n}(1-\cos\sigma)<c_{\infty}.
\end{equation}
Substituting $\sigma=0$ into \eqref{eq:4.2}, we have
\begin{equation*}\label{eq:4.3}
	\alpha<\arcsin\Big(\frac{c_{\infty}}{q_{\infty}}\Big)+\theta=\alpha_{0}+\theta.
\end{equation*}
This implies that our model is valid only for $\theta>-\alpha_{0}$.

Next, we focus on the flow state inside the Mach cone. Unlike that for the triangle plate, if the scaling transformation \eqref{eq:2.32} is carried out directly, then there is an oblique derivative condition on $\Gamma_{wing}$. Since the type of equation \eqref{eq:2.41} is a prior unknown, it is quite involved to establish the Lipschitz estimate for the solution in this case. To this end, we will first make a rotation of coordinates before performing a scaling transformation, as shown in \eqref{eq:4.4}--\eqref{eq:4.5}. Accordingly, the oblique derivative condition is reduced to a Neumann condition on $\Gamma_{wing}$. Later, we will develop our method to treat the problem \eqref{eq:3.13} with a Neumann condition on $\Gamma_{wing}\cup\Gamma_{sym}$ and a Dirichlet  condition on $\Gamma_{cone}^{\infty}\cup\Gamma_{cone}^{\sigma}$ in a Lipschitz domain. It is worth pointing out that the root chord of the wing $W_{\sigma}^{\theta}$ corresponds to a corner point in the $(\xi_1,\xi_2)$-plane (see \Cref{fg21}), which needs examining more carefully.

We consider a rotation transformation as follows
\begin{equation}\label{eq:4.4}
	(\tilde{x}_{1},\tilde{x}_{2},\tilde{x}_{3})=(x_{1}\cos\theta-x_{3}\sin\theta,x_{2},x_{1}\sin\theta+x_{3}\cos\theta).
\end{equation}
Then the corresponding scaling is given by
\begin{equation}\label{eq:4.5}
	(\tilde{x}_1,\tilde{x}_2,\tilde{x}_3)\longrightarrow(\tilde{\xi}_{1},\tilde{\xi}_{2}):=\Big(\frac{\tilde{x}_1}{\tilde{x}_3},\frac{\tilde{x}_2}{\tilde{x}_3}\Big), \quad
	(\rho,\Phi)\longrightarrow(\rho,\tilde{\phi}):=\Big(\rho, \frac{\Phi}{\tilde{x}_3}\Big),
\end{equation}
where $(\tilde{\xi}_{1},\tilde{\xi}_{2})$ are the new conical coordinates. Similarly, we introduce
\begin{equation*}
	\tilde{\psi}:= \frac{\tilde\phi}{\sqrt{B_\infty}},
\end{equation*}
where $B_\infty$ is defined by \eqref{eq1:1}. Since equation \eqref{eq:1.3} does not change under the rotation transformation \eqref{eq:4.4},  this equation can be expressed in terms of $\tilde\psi$ as
\begin{equation}\label{eq1:4.5}
	\mathrm{div}_{\tilde{\bm{\xi}}}(\rho(D_{\tilde{\bm{\xi}}}\tilde\psi-(\tilde\psi-D_{\tilde{\bm{\xi}}}\tilde\psi\cdot\tilde{\bm{\xi}})\tilde{\bm{\xi}}))+2\rho(\tilde\psi-D_{\tilde{\bm{\xi}}}\tilde\psi\cdot\tilde{\bm{\xi}})=0,
\end{equation}
where $\tilde{\bm{\xi}}:=(\tilde{\xi}_1,\tilde{\xi}_2)$, and $\mathrm{div}_{\tilde{\bm{\xi}}}$ and  $D_{\tilde{\bm{\xi}}}$ stand for the divergence and gradient operators with respect to $\tilde{\bm{\xi}}$, respectively.  Then the problem for $\tilde\psi$ is
\begin{equation}\label{eq:4.6}
	\begin{cases}
		\text{Equation}~\eqref{eq1:4.5}\quad&\text{in $\Omega$},\\
		\tilde\psi=\sqrt{1+|\tilde{\bm{\xi}}|^2}\quad&\text{on $\Gamma_{cone}^{\infty}\cup\Gamma_{cone}^{\sigma}$},\\
		D_{\tilde{\bm{\xi}}}{\tilde\psi}\cdot\tilde{\bm{\nu}}_{w}=0\quad&\text{on $\Gamma_{wing}$},\\
		D_{\tilde{\bm{\xi}}}{\tilde\psi}\cdot\tilde{\bm{\nu}}_{sy}=0\quad&\text{on $\Gamma_{sym}$},
	\end{cases}
\end{equation}
where $\tilde{\bm{\nu}}_{w}:=(1,\sin\theta\tan\sigma)$ and $\tilde{\bm{\nu}}_{sy}:=(0,-1)$ are the exterior normals to $\Gamma_{wing}$ and $\Gamma_{sym}$, respectively.

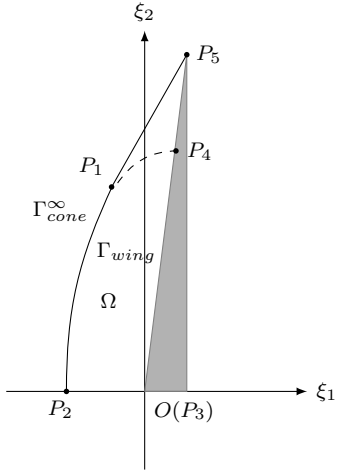
\begin{figure}[H]
	\centering
	\begin{tikzpicture}[smooth, scale=0.8]
	\draw  [-latex](1,2)--(6,2) node [right] {\footnotesize$\xi_{1}$};
	\draw  [-latex](3.3,0.7)--(3.3,8) node [above] {\footnotesize$\xi_{2}$};
	\draw  (2,2)   to [out=90, in =245] (2.78,5.45) node [above left] {\footnotesize$P_{1}$} ;
	\draw [dashed] (3.85, 6)node [right] {\footnotesize$P_{4}$} to [out=180, in =57] (2.82,5.42) ;
	\node  at (1.9,2) [below ] {\footnotesize$P_{2}$};
	\node  at (2.7,3.5) {\footnotesize$\Omega$};
	\node  at (3,4.3) {\footnotesize$\Gamma_{wing}$};
	\node  at (1.9,5) {\footnotesize$\Gamma_{cone}^{\infty}$};
	\draw (2.78,5.45)--(4,7.6) node  [right]  {\footnotesize$P_{5}$};
	\draw [draw=gray, fill=gray, fill opacity=0.6](3.3,2)--(4,7.6)--(4,2)--(3.3,2);
	\node at(3.3,2) [below right] {\footnotesize$O(P_{3})$};
	\fill(4,7.6)circle(1.3pt);
	\fill(3.82, 6)circle(1.3pt);
	\fill(2.75,5.40)circle(1.3pt);
	\fill(2,2)circle(1.3pt);
	\end{tikzpicture}
	\caption{Domain in the new conical coordinates.}
	\label{fg21}
\end{figure}

Hereafter, we will work in these new coordinates and drop $"\sim"$ for simplification.

We also use the strategy mentioned in \Cref{sec:3.2.1} to prove the existence of the solution to problem \eqref{eq:4.6}. Let us consider the following boundary value problem:
\begin{equation}\label{eq:4.7}
	\begin{cases}
		\text{Equation}~\eqref{eq:3.13}\quad&\text{in $\Omega$},\\
		\psi=\sqrt{1+|\bm{\xi}|^2}+\varepsilon\quad&\text{on $\Gamma_{cone}^{\infty}\cup\Gamma_{cone}^{\sigma}$},\\
		D\psi\cdot\bm{\nu}_{w}=0\quad&\text{on $\Gamma_{wing}$},\\
		D\psi\cdot\bm{\nu}_{sy}=0\quad&\text{on $\Gamma_{sym}$}.
	\end{cases}
\end{equation}
Still, denote by ${\psi}_{\mu,\varepsilon}$ a solution to problem \eqref{eq:4.7}. In addition, for any open bounded domain $\hat{\Omega}\subset\mathbb{R}^{2}$, we introduce the operator
\begin{equation*}
	\mathcal{M}w:=Dw\cdot\hat{\bm{\nu}}+b w, \quad\text{in}~\partial\hat{\Omega},
\end{equation*}
where $\hat{\bm{\nu}}$ is the exterior normal to $\partial\hat{\Omega}$, and $b\geq 0$ is a function of $\bm{\xi}$.  As before, we first show a comparison principle.

\begin{lemma}\label{lemma5}
	Let $\Omega_{D}\subset \mathbb{R}^{2}$ be an open bounded domain, with its boundary composed of $\partial^1\Omega_{D}$ and $\partial^2\Omega_{D}$, and let $\bm{\nu}$ be the exterior normal to $\partial^2\Omega_{D}$. Also, let ${w}_{\pm}\in C^{0}(\overline{\Omega_{D}})\cap C^{1}(\overline{\Omega_{D}}\setminus\overline{\partial^1\Omega_{D}})\cap C^2(\Omega_{D})$ satisfy ${w}_{\pm}>1$ in $\Omega_{D}$. Assume that, for any $\mu\in[0,1]$, the operator $\mathcal{N}_\mu$ is locally uniformly elliptic with respect to either ${w}_{+}$ or $w_-$, and there hold
	\begin{equation*}
		\mathcal{N}_\mu w_-\geq0, \quad\mathcal{N}_\mu w_+\leq0\quad\text{in}~\Omega_{D}
	\end{equation*}
	with $w_-\leq w_+$ on $\partial^1\Omega_{D}$ and $\mathcal{M}w_-< \mathcal{M}w_+$ on $\partial^2\Omega_{D}$, where $\mathcal{N}_\mu w=0$ denotes equation \eqref{eq:3.20}. Then, it follows that ${w}_{-}\leq{w}_{+}$ in $\Omega_{D}$.
\end{lemma}

\begin{proof}
	Setting
	\begin{equation*}
		\bar{w}:=w_- -w_+,
	\end{equation*}
	then by \Cref{lemma3} and the analysis in \Cref{sec:3.2.1}, we have
	\begin{equation*}
		\sup_{\Omega_D}\bar{w}\leq\sup_{\partial\Omega_D}\bar{w}.
	\end{equation*}
	We may as well assume that $\bar{w}$ could achieve its maximum at a point $P$ on $\partial^2\Omega_{D}$. Then we have $D\bar{w}\cdot\bm{\nu}|_{P}\geq 0$, which means $\bar{w}|_{P}<0$ due to $b\geq 0$. This leads to a contradiction. Hence we have $\bar{w}\leq 0$ in $\Omega_{D}$. The proof is completed.
\end{proof}

Now we are in a position to derive a prior estimates for ${\psi}_{\mu,\varepsilon}$. Similarly, we apply the corresponding function  $w_{\mu,\varepsilon}$. From \eqref{eq:3.3} and \eqref{eq:4.7}, we obtain the equation for $w_{\mu,\varepsilon}$ as follows
\begin{equation}\label{eq:4.16}
	\begin{cases}
		\mathcal{N}_\mu w=0\quad&\text{in $\Omega$},\\
		w=1+\varepsilon\quad&\text{on $\Gamma_{cone}^{\infty}\cup\Gamma_{cone}^{\sigma}$},\\
		Dw\cdot\bm{\nu}_{w}=0\quad&\text{on $\Gamma_{wing}$},\\
		Dw\cdot\bm{\nu}_{sy}=0\quad&\text{on $\Gamma_{sym}$}.
	\end{cases}
\end{equation}

Before proceeding, let us define two domains
\begin{align}
	\Lambda_1:&=\{\arctan(\cot\sigma/\tan\theta)<\vartheta<\pi\},\label{eq4:14}\\
	\Lambda_2:&=\{\arctan(\cot\sigma/\tan\theta)+\pi<\vartheta<2\pi\},\label{eq4:15}
\end{align}
where $\vartheta=\arctan(\xi_2/\xi_1)\in [0,2\pi)$. For simplicity, the subscripts of $\psi_{\mu,\varepsilon}$ and $w_{\mu,\varepsilon}$ will be omitted in the proof of \Cref{lemma:inf-lip estimates,lemma: lip} below.

\begin{lemma}\label{lemma:inf-lip estimates}
	Let ${w}\in C^{0}(\bar{\Omega})\cap C^{1}(\bar{\Omega}\setminus\overline{\Gamma_{cone}^{\infty}\cup\Gamma_{cone}^{\sigma}})\cap C^2(\Omega\cup\Gamma_{sym}\cup\Gamma_{wing})$ be a solution of problem \eqref{eq:4.16} with $w> 1$ in $\Omega$. Then there exists a constant $C$, independent of $\mu$ and $\varepsilon$, such that
	\begin{equation*}\label{eq:4.17}
		1+\varepsilon<{w}\leq C\quad\text{in $\bar{\Omega}\setminus\overline{\Gamma_{cone}^{\infty}\cup\Gamma_{cone}^{\sigma}}$}.
	\end{equation*}
\end{lemma}

\begin{proof}
	We first analyze the function ${w}^{\bm{\eta}}$, which is an exact solution to the equation $\mathcal{N}_\mu w=0$. From \eqref{eq1:3.26}, we have
	\begin{equation}\label{eq4:1}
		D{w}^{\bm{\eta}}(\bm{\xi})\cdot \bm{\nu}_{w} =\frac{\tilde{\bm{\eta}}\cdot \bm{\nu}_{w}}{\sqrt{1+|\bm{\xi}|^{2}}},\quad D{w}^{\bm{\eta}}(\bm{\xi})\cdot \bm{\nu}_{sy} =\frac{\tilde{\bm{\eta}}\cdot \bm{\nu}_{sy}}{\sqrt{1+|\bm{\xi}|^{2}}},
	\end{equation}
	where $\bm{\eta}=(\tilde{\bm{\eta}},\eta_3)=(\eta_1,\eta_2,\eta_3)$. Let $O^{\tilde{\bm{\eta}}}=(\xi_1,\xi_2)$ denote the corresponding coordinates of $\tilde{\bm{\eta}}$ with $\eta_3=1$. Note that $\arctan(\cot\sigma/\tan\theta)$ is the included angle of the line $P_{3}P_{5}$ and the $\xi_{1}$-axis (see \Cref{fg21}). Then we infer from \eqref{eq4:1} that when $O^{\tilde{\bm{\eta}}}\in \Lambda_1$, both $Dw^{\bm{\eta}}\cdot\bm{\nu}_{w}$ on $\Gamma_{wing}$ and $Dw^{\bm{\eta}}\cdot\bm{\nu}_{sy}$ on $\Gamma_{sym}$ are negative; when $O^{\tilde{\bm{\eta}}}\in \Lambda_2$, however, both of them are positive.
	
	Next, we construct the sub- and super-solutions to problem \eqref{eq:4.16}. For the super-solution, let us consider the set with a fixed $\varepsilon>0$:
	\begin{equation}\label{eq:4.8}
		\Sigma_{+}^{\theta}:=\{\bm{\eta}\in\mathbb{R}^{3}: O^{\tilde{\bm{\eta}}}\in \Lambda_2, \text{ and }{w}^{\bm{\eta}}>1+\varepsilon\text{ on }\Gamma_{cone}^{\infty}\cup\Gamma_{cone}^{\sigma}\}.
	\end{equation}
	When $\bm{\eta}\in \Sigma_{+}^{\theta}$, the function ${w}^{\bm{\eta}}$    satisfies
	\begin{equation*}\label{eq1:4.16}
		\begin{cases}
			\mathcal{N}_\mu w^{\bm{\eta}}= 0\quad&\text{in $\Omega$},\\
			w^{\bm{\eta}}>1+\varepsilon\quad&\text{on $\Gamma_{cone}^{\infty}\cup\Gamma_{cone}^{\sigma}$},\\
			Dw^{\bm{\eta}}\cdot\bm{\nu}_{w}>0\quad&\text{on $\Gamma_{wing}$},\\
			Dw^{\bm{\eta}}\cdot\bm{\nu}_{sy}>0\quad&\text{on $\Gamma_{sym}$}.
		\end{cases}
	\end{equation*}
	Recall that for any fixed constant vector $\bm{\eta}\in \mathbb{R}^{3}$, the function $w^{\bm{\eta}}$ decreases with respect to the angle between $\frac{(\xi_{1},\xi_{2},1)}{\sqrt{1+|\bm{\xi}|^{2}}}$ and $\bm{\eta}$. Then from \eqref{eq:3.26}--\eqref{eq1:3.26} and \eqref{eq:4.8}, it follows that the operator $\mathcal{N}_\mu $ is locally uniformly elliptic with respect to $w^{\bm{\eta}}$. Thanks to \Cref{lemma5}, we have $w\leq w^{\bm{\eta}}$ in $\Omega$. Let $w^+$ be the infimum of all these $w^{\bm{\eta}}$; then $w\leq w^+$ in $\Omega$.
	Moreover, owing to the convexity of  $\Gamma_{cone}^{\infty}\cup\Gamma_{cone}^{\sigma}$, we know that $w^+=1+\varepsilon$ on this boundary.
	
	For the sub-solution, we consider the set with a fixed $\varepsilon>0$:
	\begin{equation}\label{eq:4.9}
		\Sigma_{-}^{\theta}:=\{\bm{\eta}\in\mathbb{R}^{3}: {w}^{\bm{\eta}}<1+\varepsilon\text{ on }\partial\Omega\},
	\end{equation}
	and define
	\begin{equation*}
		{w}^{-}:=\sup\{{w}^{\bm{\eta}}:\bm{\eta}\in \Sigma_{-}^{\theta}\}.
	\end{equation*}
	From \eqref{eq:3.2} and \eqref{eq:4.9}, we find that the set $\Sigma_{-}^{\theta}$ is the same as the set $\Sigma_{-}$ by replacing the boundary $\partial\Omega$ by $\partial\Omega_{ext}$. Then proceeding as in the argument of \Cref{lemma: inf-estimate}, we have $w^->1+\varepsilon$ in $\Omega$ and $w^-=1+\varepsilon$ on $\partial\Omega$. Moreover, for any compact subdomain $\Omega^c_{sub}\subset\Omega$, we have
	\begin{equation}\label{eq1:12}
		{w}^{-}> 1+\varepsilon+\delta\quad\text{in $\Omega^c_{sub}$}
	\end{equation}
	with the constant $\delta$ only depending on $\Omega^c_{sub}$. This means that the operator $\mathcal{N}_\mu$ is locally uniformly elliptic with respect to $w^-$. Notice that the function $w^-$ is less than or equal to $1+\varepsilon$ outside $\Omega$. Then, there hold $Dw^-\cdot\bm{\nu}_{w}<0$ on $\Gamma_{wing}$ and $Dw^-\cdot\bm{\nu}_{sy}<0$ on $\Gamma_{sym}$. From \Cref{lemma5}, we know
	\begin{equation}\label{eq1:11}
		{w}\geq {w}^{-}> 1+\varepsilon\quad\text{in $\Omega$}.
	\end{equation}
	
	For the point $P_3$ and the interior points of $\Gamma_{wing}$ and $\Gamma_{sym}$, we only have the estimate ${w}\geq w^-=1+\varepsilon$. When $\varepsilon=0$, this estimate cannot rule out the existence of parabolic bubbles on the boundaries $\Gamma_{wing}$ and $\Gamma_{sym}$ or at the point $P_3$.  Hence we need to improve the estimate for these points. To this end, let us consider the set with a fixed $\varepsilon>0$:
	\begin{equation*}
		\widetilde{\Sigma}_{-}^{\theta}:=\{\bm{\eta}\in\mathbb{R}^{3}: O^{\tilde{\bm{\eta}}}\in \Lambda_1, \text{ and } {w}^{\bm{\eta}}<1+\varepsilon\text{ on }\Gamma_{cone}^{\infty}\cup\Gamma_{cone}^{\sigma}\}.
	\end{equation*}
	Obviously, when $\bm{\eta}\in \widetilde{\Sigma}_{-}^{\theta}$, the function ${w}^{\bm{\eta}}$ satisfies
	\begin{equation*}
		\begin{cases}
			\mathcal{N}_\mu w^{\bm{\eta}}= 0\quad&\text{in $\Omega$},\\
			w^{\bm{\eta}}<1+\varepsilon\quad&\text{on $\Gamma_{cone}^{\infty}\cup\Gamma_{cone}^{\sigma}$},\\
			Dw^{\bm{\eta}}\cdot\bm{\nu}_{w}<0\quad&\text{on $\Gamma_{wing}$},\\
			Dw^{\bm{\eta}}\cdot\bm{\nu}_{sy}<0\quad&\text{on $\Gamma_{sym}$}.
		\end{cases}
	\end{equation*}
	It follows from \eqref{eq1:12} and \eqref{eq1:11} that the operator $\mathcal{N}_\mu $ is locally uniformly elliptic with respect to $w$. In addition, if $\bm{\eta}\in \widetilde{\Sigma}_{-}^{\theta}$, then $w^{\bm{\eta}}$ is a sub-solution to problem \eqref{eq:4.16} in the subdomain $\Omega_{sub}$ where $w^{\bm{\eta}}$ is larger than $1$. By \Cref{lemma3} or \Cref{lemma5}, we have the estimate $w\geq w^{\bm{\eta}}$ in the subdomain $\Omega_{sub}$. Define $\tilde{w}^-$ as the supermum of all these $w^{\bm{\eta}}$; then $w\geq \tilde{w}^-$ in $\Omega$. Using the argument as in the proof of \eqref{eq1:3.21}, we can find a constant $\delta>0$, independent of $\mu$ and $\varepsilon$, such that $\tilde{w}^-\geq 1+\varepsilon+\delta$ at the point $P_3$ and the interior point of $\Gamma_{wing}$ and $\Gamma_{sym}$. Thus, from the discussion above, we have
	\begin{align}
		1+\varepsilon<\tilde{w}^-\leq w\leq w^+\quad&\text{in}~  \bar{\Omega}\setminus\overline{\Gamma_{cone}^{\infty}\cup\Gamma_{cone}^{\sigma}}\label{eq1:4.19}
		\shortintertext{and}
		w=w^\pm=1+\varepsilon\quad&\text{on}~ \Gamma_{cone}^{\infty}\cup\Gamma_{cone}^{\sigma}.\label{eq1:4.20}
	\end{align}
	This complete the proof.
\end{proof}

From estimates \eqref{eq1:4.19} and \eqref{eq1:4.20}, it follows that
\begin{align}
	\sqrt{1+|\bm{\xi}|^2}+\varepsilon<{\psi}\leq C\quad&\text{in}~\bar{\Omega}\setminus\overline{\Gamma_{cone}^{\infty}\cup\Gamma_{cone}^{\sigma}}\label{eq1:4.17}\\
	\shortintertext{and}
	\psi=\psi^\pm\quad&\text{on}~\Gamma_{cone}^{\infty}\cup\Gamma_{cone}^{\sigma},\label{eq1:4.18}
\end{align}
where $\psi^\pm=\sqrt{1+|\bm{\xi}|^2}w^\pm$, and $C$ is a constant independent of $\mu$ and $\varepsilon$.

\begin{lemma}\label{lemma: lip}
	Let $\psi \in C^{0}(\bar{\Omega})\cap C^{1}(\bar{\Omega}\setminus\overline{\Gamma_{cone}^{\infty}\cup\Gamma_{cone}^{\sigma}})\cap C^2(\Omega\cup\Gamma_{sym}\cup\Gamma_{wing})$ satisfy problem \eqref{eq:4.7} with $\psi> \sqrt{1+|\bm{\xi}|^2}$ in $\Omega$. Then there exists a constant $C>0$, independent of $\mu$ and $\varepsilon$, such that
	\begin{equation}\label{eq1:4}
		\|D{\psi}\|_{L^{\infty}(\Omega)}\leq C.
	\end{equation}
\end{lemma}

\begin{proof}
	From the discussion of interior Lipschitz estimates in \Cref{lemma: lip-estimate}, we have
	\begin{equation}\label{eq:4.11}
		\|D{\psi}\|_{L^{\infty}(\Omega)}\leq \|D{\psi}\|_{L^{\infty}(\partial\Omega)}.
	\end{equation}
	Since problem \eqref{eq:4.7} is invariant under a rotation transformation and is also of reflection symmetry with respect to the straight boundaries $\Gamma_{wing}$ and $\Gamma_{sym}$, it follows that any point on $\Gamma_{wing}$ and $\Gamma_{sym}$ can be treated as an interior point of the domain by an even extension (for example, see \cite{CQ12}). In addition, we find that the normal vectors of $\Gamma_{wing}$ and $\Gamma_{sym}$ are different. Then $D{w}=0$ at the point $P_3$ follows from the boundary conditions $D{w}\cdot \bm{\nu}_{sy}=0$ on $\Gamma_{sym}$ and $D{w}\cdot \bm{\nu}_{w}=0$ on $\Gamma_{wing}$. Consequently, estimate \eqref{eq:4.11} is reduced to
	\begin{equation*}
		\|D{\psi}\|_{L^{\infty}(\Omega)}\leq \|D{\psi}\|_{L^{\infty}(\Gamma_{cone}^{\infty}\cup\Gamma_{cone}^{\sigma})}.
	\end{equation*}
	Using the argument as in \Cref{lemma: lip-estimate}, along with \eqref{eq1:4.17} and \eqref{eq1:4.18}, we obtain the boundedness of $\|D{\psi}\|_{L^{\infty}(\Gamma_{cone}^{\infty}\cup\Gamma_{cone}^{\sigma})}$. Therefore, this lemma is proved.
\end{proof}

According to \Cref{lemma:inf-lip estimates,lemma: lip}, we now define
\begin{equation*}\label{eq4:11}
	\begin{aligned}
		J^\theta_{\varepsilon}:=\{\mu\in[0,1]:~&\text{such that }\psi_{\mu,\varepsilon} ~\text{satisfies}~\eqref{eq:4.7}\text{ with}~\psi_{\mu,\varepsilon}\geq \sqrt{1+\vert\bm{\xi}\vert^2}+\varepsilon ~\text{in}~\Omega~\text{and}\\
		&\psi_{\mu,\varepsilon}\in C^{0}(\bar{\Omega})\cap C^{1}(\bar{\Omega}\setminus\overline{\Gamma_{cone}^{\infty}\cup\Gamma_{cone}^{\sigma}})\cap C^2(\Omega\cup\Gamma_{sym}\cup\Gamma_{wing})\}.
	\end{aligned}
\end{equation*}

We first prove that for any fixed $\varepsilon>0$, the set $J^\theta_{\varepsilon}$ is closed. When $\mu\in J^\theta_{\varepsilon}$, by \Cref{lemma:inf-lip estimates,lemma: lip}, the function $\psi_{\mu,\varepsilon}$ satisfies estimates
\eqref{eq1:4.17} and \eqref{eq1:4}. Hence the corresponding linearized equation \eqref{eq:3.45} is uniformly elliptic and $\psi_{\mu,\varepsilon}\in Lip(\bar{\Omega})\cap C^{1}(\bar{\Omega}\setminus\overline{\Gamma_{cone}^{\infty}\cup\Gamma_{cone}^{\sigma}})\cap C^2(\Omega\cup\Gamma_{sym}\cup\Gamma_{wing})$. The crucial point here is to improve the regularity of $\psi_{\mu,\varepsilon}$ at the corner point $P_3$. To this end, we apply Lemma 1.3 in \cite{Lieb88} to the neighborhood of the point $P_3$. Then, there exists a positive constant $\kappa=\kappa(\theta,\sigma)\in(0,1)$ such that $|\psi_{\mu,\varepsilon}|^{-1-\kappa}_{2}\leq C$, where the constant $C$ only depends on $\Omega$. For the norm $|\cdot|^{-1-\kappa}_{2}$, we refer the reader to \cite{GH,Lieb88}. Owing to $\|\psi_{\mu,\varepsilon}\|_{c^{1,\kappa}}\leq|\psi_{\mu,\varepsilon}|^{-1-\kappa}_{2}$, the regularity of $\psi_{\mu,\varepsilon}$ at the point $P_3$ is $C^{1,\kappa}$ for $\kappa\in (0,1)$.

It remains to verify that Lemma 1.3 in \cite{Lieb88} is valid for the problem \eqref{eq:3.45} with the boundary conditions as in \eqref{eq:4.7}.  When $\mu\in J^\theta_{\varepsilon}$, the coefficients of equation \eqref{eq:3.45} belongs to $C^{0}(\bar{\Omega}\setminus\overline{\Gamma_{cone}^{\infty}\cup\Gamma_{cone}^{\sigma}})\cap C^1(\Omega)$. In addition, this equation does not have lower order terms, so the conditions $(1.4a)-(1.5d)$ in \cite{Lieb88} are naturally satisfied. Also, since equation \eqref{eq:3.45} is uniformly elliptic and the angle at $P_3$ is equal to $\pi-\arctan(\cot\sigma/\tan\theta)$, there also hold the conditions $(1.6a)-(1.7)$, as required in \cite{Lieb88}.

Furthermore, as mentioned above, any point on $\Gamma_{wing}\cup\Gamma_{sym}$ can be treated as an interior point of the domain. Then it follows form  Theorem 6.17 in \cite{GT01} that $\psi_{\mu,\varepsilon}$ is bounded in $Lip(\bar{\Omega})\cap C^{1,\kappa}(\bar{\Omega}\setminus\overline{\Gamma_{cone}^{\infty}\cup\Gamma_{cone}^{\sigma}})\cap C^\infty(\Omega\cup\Gamma_{sym}\cup\Gamma_{wing})$, which indicates that $J^\theta_{\varepsilon}$ is closed, as shown in \Cref{sec:3.2.5}.

The remaining part of the proof can be completed in much the same way as in \Cref{sec:3.2.5}, so we just omit the details.

We summarize this section by stating the following theorem.

\begin{theorem}\label{thm5}
	Assume that the state $(\rho_{\infty},q_{\infty})$ of the oncoming flow is uniform and supersonic, and the thin wing $W_{\sigma}^{\theta}$ is defined by \eqref{eq4:12}. Then we can find a critical angle $\alpha_0=\alpha_0(\rho_\infty, q_\infty)\in(0,\pi/2)$ so that if $\theta\in(-\alpha_0,0)$, then for any $\alpha\in(0,\alpha_0+\theta)$, there exists $\sigma_{0}=\sigma_0(\rho_\infty, q_\infty,\alpha)\in(0,\pi/2)$ such that, when $\sigma\in[0,\sigma_0]$, there exists a constant  $\kappa=\kappa(\theta,\sigma)\in (0,1)$ and the problem \eqref{eq:1.3}--\eqref{eq:1.4} with \eqref{eq1:1.1}, \eqref{eq1:1.2}, \eqref{eq1:41} admits a piecewise smooth solution
	\begin{equation*}
		\Phi(\tilde{\bm{x}})=\sqrt{B_\infty}\tilde{x}_3\psi\Big(\frac{\tilde{\bm{x}}}{\tilde{x}_3}\Big)
	\end{equation*}
	in the domain $\mathcal{R}_{\sigma}^{\theta}$, satisfying
	\begin{equation*}
		\psi\in Lip(\bar{U})\cap C^{1,\kappa}(\bar{U}\setminus\overline{\Gamma_{cone}^{\infty}\cup\Gamma_{cone}^{\sigma}})\cap C^\infty\big(\bar{U}\setminus(\overline{\Gamma_{cone}^{\infty}\cup\Gamma_{cone}^{\sigma}}\cup\{P_3\})\big)
	\end{equation*}
	and
	\begin{equation*}
		\psi >\sqrt{1+|\tilde{\bm{\xi}}|^{2}} \quad\text{in}~ \bar{\Omega}\setminus\overline{\Gamma_{cone}^{\infty}\cup\Gamma_{cone}^{\sigma}}.
	\end{equation*} 	
	Here $\alpha_0$ and $\sigma_0$ are given by \eqref{eq:2.6} and \eqref{eq:2.7}, respectively, and the constant $B_\infty$ is defined by \eqref{eq1:1}.
\end{theorem}

\section{Problems involving non-convex domains}\label{sec:5}
\subsection{Pressure wave detached from the leading edges}\label{sec:5.1}
We now give a brief exposition of the case $\sigma\in(\sigma_{0},{\pi}/{2})$, where $\sigma_{0}$ is given by \eqref{eq:2.7}. Let $\Gamma_{cone}^{\infty}$ and $\Gamma_{wing}$ be the pressure wave and the surface of the delta wing, respectively. Since $\Gamma_{cone}^{\infty}$ is a characteristic determined by the oncoming flow, the pressure wave will be detached from the leading edges of the wing if $\sigma\in(\sigma_{0},{\pi}/{2})$ (see \Cref{fg9}), that is, the pressure wave is only attached to the apex but away from the leading edges. Although
there may exist a region of cavitation in the expansion region, it is not difficult to imagine that, once $\sigma>\sigma_{0}$, the flow will spread freely from one side into the other side through the leading edges. Then the phenomenon of cavitation disappears suddenly. In other words, this problem can still be discussed in the case $\sigma\in(\sigma_{0},{\pi}/{2})$.

Since the location of $\Gamma_{cone}^{\infty}$ determined by \eqref{eq1:1.7} is unchanged as $\sigma$ increases from $\sigma_0$ to ${\pi}/2$, we only need to focus on the state of the flow behind $\Gamma_{cone}^{\infty}$. In comparison with the case $\sigma\in(0,\sigma_{0}]$, there arise some new difficulties for $\sigma\in(\sigma_{0},\pi/2)$. Firstly, the boundary condition on $\Gamma_{wing}$ cannot be reduced to the Neumann condition as in \Cref{sec:4}. Secondly, the domain bounded by $\Gamma_{cone}^{\infty}$ and $\Gamma_{wing}$ becomes non-convex. Hence we need other treatments to obtain a prior estimates for the solution.  Finally, we may not expect the regularity of the solution to be $C^{0,1}$ at the points $P_{5}$ and $P'_{5}$ in \Cref{fg9}. Then it will be rather difficult to determine the type of equation \eqref{eq:2.41} behind the pressure wave $\Gamma_{cone}^{\infty}$.

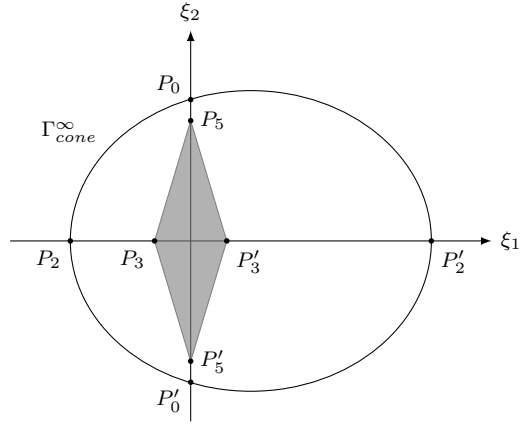
\begin{figure}[H]
	\centering
	\begin{tikzpicture}[smooth, scale=0.8]
	\draw (5,4) ellipse (3 and 2.5);
	\draw [-latex](1,4)--(9,4)node[right]{\footnotesize$\xi_{1}$};
	\draw [-latex](4,1)--(4,7.5)node[above]{\footnotesize$\xi_{2}$};
	\draw [draw=gray, fill=gray, fill opacity=0.6](4,2)--(3.4,4)--(4,6)--(4.6,4)--(4,2);
	\node at (3.4,4)[below left]{\footnotesize $P_{3}$};
	\node at (4,6)[right]{\footnotesize $P_{5}$};
	\node at (4.6,4)[below right]{\footnotesize $P'_{3}$};
	\node at (4,2)[right]{\footnotesize $P'_{5}$};
	\node [below left, font=\footnotesize] at (2,4) {$P_{2}$};
	\node [below left, font=\footnotesize] at (4,1.7) {$P'_{0}$};
	\node[above left, font=\footnotesize] at (4,6.3) {$P_{0}$};
	\node [below right, font=\footnotesize] at (8,4) {$P'_{2}$};
	\node [above left, font=\footnotesize] at (2.6,5.5) {$\Gamma_{cone}^{\infty}$};
	\fill(4,1.65)circle(1.3pt);
	\fill(4,6.35)circle(1.3pt);
	\fill(2,4)circle(1.3pt);
	\fill(8,4)circle(1.3pt);
	\fill(4,6)circle(1.3pt);
	\fill(4,2)circle(1.3pt);
	\fill(3.4,4)circle(1.3pt);
	\fill(4.6,4)circle(1.3pt);
	\end{tikzpicture}
	\caption{Nonconvex domain produced by the detached pattern.}
	\label{fg9}
\end{figure}

\subsection{Delta wings of asymmetric cross-sections}
In the previous sections, we always assume that the delta wing is symmetrical about the $x_{1}Ox_{3}$-plane. So it is also natural to consider an asymmetric one. As before, we only discuss the case of attached shocks and adopt the previous notations as well.

\begin{figure}[H]
	\centering
	\subfigure[Asymmetric triangular plate with a convex domain.]{
		\begin{tikzpicture}[smooth, scale=0.8]
		\draw  [-latex](0.5,2)--(5.5,2) node [right] {\footnotesize$\xi_{1}$};
		\draw  [-latex](4,-1)--(4,6.5) node [above] {\footnotesize$\xi_{2}$};
		\draw  (2.3,3.1)--(4,6) node  [right]  {\footnotesize$P_{5}$};
		\draw  (2,2)   to [out=90, in =240] (2.3,3.1) node [above left] {\footnotesize$P_{1}$} ;
		\draw [dashed] (2.3,3.1) to [out=50, in =180]  (4,4) node [right] {\footnotesize$P_{4}$};
		\draw (3.2,-0.02)--(4,-0.5);
		\draw (2,2) to [out=-90,in=144] (3.2,-0.02);
		\draw [dashed] (3.2,-0.02)  to [out=-36, in =180] (4,-0.2) ;
		\node at (4.4,-0.6) {\footnotesize$P'_{5}$} ;
		\node at (4.4,0) {\footnotesize$P'_{4}$} ;
		\node at (2,2) [below left] {\footnotesize$P_{2}$};
		\node at (4,2) [below right] {\footnotesize$O(P_{3})$};
		\node at (4.6,3) {\footnotesize$\Gamma_{wing}$};
		\node at (1.2,2.5) {\footnotesize$\Gamma_{cone}^{\infty}$};
		\node at (3.2,-0.02)[below left] {\footnotesize$P'_{1}$};
		\node at (3.2,2.4){\footnotesize$\Omega$};
		\draw (4,6)--(4,0);
		\fill(4,4)circle(1.3pt);
		\fill(2.3,3.1)circle(1.3pt);
		\fill(2,2)circle(1.3pt);
		\fill(4,-0.5)circle(1.3pt);
		\fill(4,-0.2)circle(1.3pt);
		\fill(3.2,-0.02)circle(1.3pt);
		\fill(4,6)circle(1.3pt);
		\end{tikzpicture}
		\label{fg22}
	}\qquad
	\subfigure[Asymmetric thin delta wing with a non-convex domain.]{
		\begin{tikzpicture}[smooth, scale=0.8]
		\draw  [-latex](0.5,2)--(5.5,2) node [right] {\footnotesize$\xi_{1}$};
		\draw  [-latex](4,-1)--(4,6.5) node [above] {\footnotesize$\xi_{2}$};
		\draw  (2.3,3.1)--(4,6) node  [right]  {\footnotesize$P_{5}$};
		\draw  (2,2)   to [out=90, in =240] (2.3,3.1) node [above left] {\footnotesize$P_{1}$} ;
		\draw [dashed] (2.3,3.1) to [out=50, in =170]  (3.62,3.8) ;
		\draw (3.2,-0.02)--(4,-0.5);
		\draw (2,2) to [out=-90,in=144] (3.2,-0.02);
		\draw [dashed] (3.2,-0.02)  to [out=-36, in =190] (3.89,-0.16) ;
		\node at  (4,3.8) [right] {\footnotesize$P_{4}$} ;
		\node at (4.4,-0.6) {\footnotesize$P'_{5}$} ;
		\node at (4.4,0) {\footnotesize$P'_{4}$} ;
		\node at (2,2) [below left] {\footnotesize$P_{2}$};
		\node at (4,2) [below right] {\footnotesize$O$};
		\node at (3,1) {\footnotesize$\Gamma_{wing}$};
		\node at (1.2,2.5) {\footnotesize$\Gamma_{cone}^{\infty}$};
		\node at (3.2,-0.02)[below left] {\footnotesize$P'_{1}$};
		\draw (4,6)--(4,0);
		\draw [draw=gray, fill=gray, fill opacity=0.6] (3.3,2)--(4,6)--(4,-0.5)--(3.3,2);
		\node at (3.45,1.5) [above left] {\footnotesize$P_{3}$};
		\node at (2.7,2.4){\footnotesize$\Omega$};
		\fill(3.62,3.8)circle(1.3pt);
		\fill(2.3,3.1)circle(1.3pt);
		\fill(2,2)circle(1.3pt);
		\fill(4,-0.5)circle(1.3pt);
		\fill(3.89,-0.16)circle(1.3pt);
		\fill(3.2,-0.02)circle(1.3pt);
		\fill(4,6)circle(1.3pt);
		\end{tikzpicture}
		\label{fg12}
	}
	\caption{Connection between ``symmetry'' and ``convexity''.}
\end{figure}
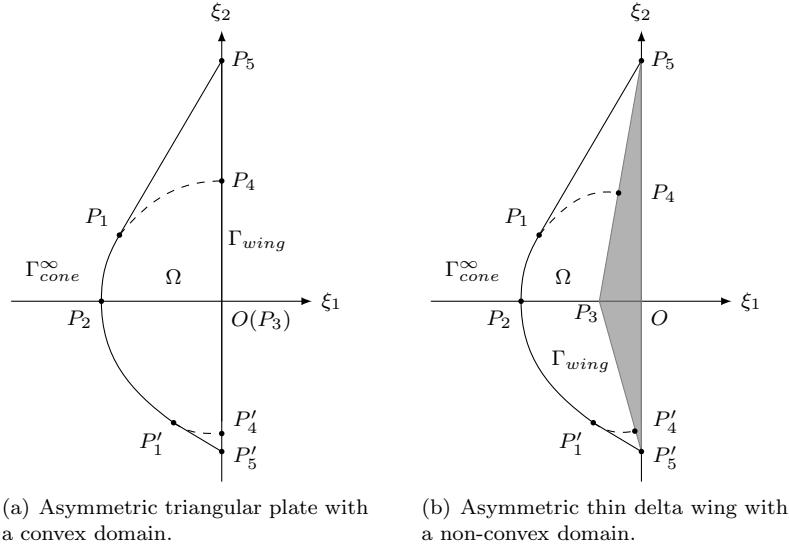

Let us begin with an asymmetric triangular plate $W_{\sigma,\hat{\sigma}}$, given by
\begin{equation*}
	W_{\sigma,\hat{\sigma}}=\{(x_1,x_2,x_3):-x_{3}\cot\hat{\sigma} <x_{2}<x_{3}\cot\sigma,x_{1}=0,x_{3}>0\},
\end{equation*}
where $\sigma,\hat{\sigma}\in(0,{\pi}/{2})$. For any fixed $\alpha\in(0,\alpha_0)$, as what we did in \Cref{sec:2}, we can derive the location of the shock and the uniform flow state outside the Mach cone when $\sigma$ and $\hat{\sigma}$ are less than a certain critical angle. Then we examine the flow inside the Mach cone. In the conical coordinates, since the corresponding domain $\Omega$ can be extended to a convex domain by an even reflection with respect to the $\xi_2$-axis (see \Cref{fg22}), we can obtain the flow state by applying the method in \Cref{sec:3}. Therefore, the same result as \Cref{thm: Main} holds for the asymmetric triangular plate.

However, for a thin wing of asymmetric cross-sections (see \Cref{fg12}), the domain $\Omega$ becomes completely non-convex, and thus we encounter the same difficulties as in \Cref{sec:5.1}. We will consider these problems in the future.

\appendix
\section{Shock polar for a Chaplygin gas}\label{sec:appendix a}
It is shown in \cite{Serre11} that a pressure wave between two constant states must be tangent to the sonic circles, whose centers are the velocities and radii are the sound speeds.  With this property, we now characterize the shock polar for a Chaplygin gas, by studying the model of supersonic flow around a corner.

\begin{figure}[H]
	\centering
	\subfigure[Showing an oblique shock.]{
		\begin{tikzpicture}[scale = 1.05, smooth]
		\draw  [-latex](-1.4,0)--(3.5,0) node [right] {\footnotesize$x_{1}$};
		\draw  [-latex](0,-1)--(0,2) node [above] {\footnotesize$x_{2}$};
		\node [below right, font=\footnotesize] at (0,0) {$O$};
		\filldraw [draw=gray, fill=gray, fill opacity=0.5] (0,0)--(3,1)--(3,1|-0,0)--cycle;
		\draw (0,0)--(2.3,1.8) node[right, font=\footnotesize] {$\mathcal{S}$};
		\draw [thick, -stealth] (-1.2,0.3)--(-0.3,0.3);
		\node[above, font=\footnotesize] at (-0.7,0.3) {$(u_{0},0)$} ;
		\draw [thick, -stealth] (1,0.65)--(1.6,0.85) node[above right, font=\footnotesize] {$(u_{1},v_{1})$};
		\draw (0:0.75) arc (0:19:0.75);
		\node [right, font=\footnotesize] at (0.75,0.15) {$\alpha$};
		\draw (0:0.4) arc (0:40:0.4);
		\node [right, font=\footnotesize] at (0.35,0.3) {$\beta$};
		\end{tikzpicture}
		\label{fg2}
	}\quad
	\subfigure[Showing an oblique rarefaction wave.]{
		\begin{tikzpicture}[scale = 1.05, smooth]
		\draw  [-latex](-1.4,1)--(3,1) node [right] {\footnotesize$x_{1}$};
		\draw  [-latex](0,-0.5)--(0,2.6) node [above] {\footnotesize$x_{2}$};
		\node [below left, font=\footnotesize] at (0,1) {$O$};
		\filldraw [draw=gray, fill=gray, fill opacity=0.5] (-1.38,1)--(0,1)--(3,0)--(3,0-|-1.38,0)--cycle;
		\draw (0,1)--(1.9,2.5) node [right, font=\footnotesize] {$\mathcal{S}$};
		\draw [thick, -stealth] (-1.2,1.3)--(-0.3,1.3);
		\node[above, font=\footnotesize] at (-0.7,1.3) {$(u_{0},0)$} ;
		\draw [thick, -stealth] (1.3,0.9)--(2.2,0.6) node[right, font=\footnotesize] {$(u'_{1},v'_{1})$} ;
		\begin{scope}[yshift=1cm]
		\draw (0:0.75) arc (0:-19:0.75);
		\end{scope}
		\node [right, font=\footnotesize] at (0.65,0.85) {$\alpha'$};
		\begin{scope}[yshift=1cm]
		\draw (0:0.4) arc (0:40:0.4);
		\end{scope}
		\node [right, font=\footnotesize] at (0.35,1.2) {$\beta$};
		\end{tikzpicture}
		\label{fg5}
	}
	\caption{Supersonic flow around a corner.}
\end{figure}
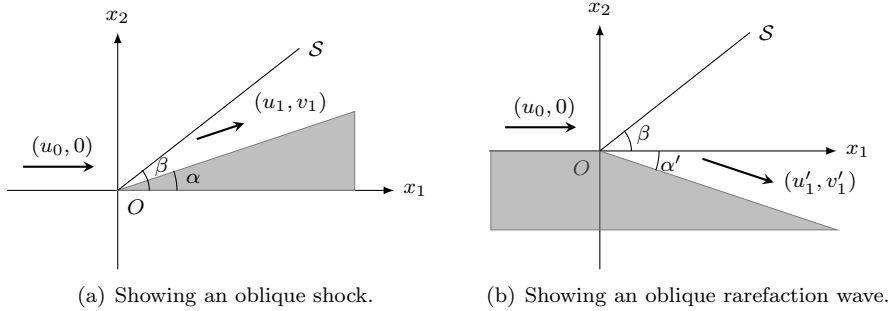

Let us first discuss the case of shocks. We assume that a shock line $\mathcal{S}$ in the $(x_{1},x_{2})$-plane is straight and passes through the origin $O$, and moreover the oncoming flow state $(\rho_{0},(u_{0},0))$ and the outgoing flow state $(\rho_{1},(u_{1},v_{1}))$ on the both sides are constant (see \Cref{fg2}). Given $(\rho_{0},(u_{0},0))$ with $u_{0}>c_{0}$ and $c_0={\sqrt{A}}/{\rho_{0}}>0$, we need to derive $(\rho_{1},(u_{1},v_{1}))$ under the condition
\begin{equation}\label{eq:A.1}
	\tan\alpha=\frac{v_{1}}{u_{1}}.
\end{equation}

In the $(u,v)$-plane, let $C_{0}$ be the circle with center $O_{0}(u_{0},0)$ and radius $c_{0}$, and $C_{1}$ the circle with center $O_{1}(u_{1},v_{1})$ and radius $c_{1}=\sqrt{A}/\rho_{1}$. It follows from the property mentioned above that the shock line $\mathcal{S}$ is tangent to the circle $C_{0}$ at a point $P$. Then
\begin{equation}\label{eq:A.2}
	\sin\beta=\frac{c_{0}}{u_{0}},
\end{equation}	
where $\beta>0$ is the angle between the shock line $\mathcal{S}$ and the velocity of the oncoming flow. To get the state $(\rho_{1},(u_{1},v_{1}))$, we first determine the position of $O_{1}$. Note that the circle $C_{1}$ is also tangent to the shock line $\mathcal{S}$ at the point $P$. This can be done immediately by using \eqref{eq:A.1}, as shown in \Cref{fg3}. Then, from \eqref{eq:A.1} and \eqref{eq:A.2}, we obtain the following explicit expressions:
\begin{align}
	c_{1}&=\frac{c_0-\tan\alpha\sqrt{u^2_0-c^2_0}}{c_0\tan\alpha+\sqrt{u^2_0-c^2_0}}\sqrt{u^{2}_{0}-c^{2}_{0}}\label{eq1:A.2}\\
	\shortintertext{and}
	u_{1}&=\frac{u_{0}\sqrt{u^2_0-c^2_0}}{\sqrt{u^2_0-c^2_0}+c_{0}\tan\alpha},\quad v_{1}=\frac{u_{0}\tan\alpha\sqrt{u^2_0-c^2_0}}{\sqrt{u^2_0-c^2_0}+c_{0}\tan\alpha}.\label{eq:A.3}
\end{align}	
These equations show that, for a given state $(\rho_{0},(u_{0},0))$, the angle $\alpha$ determines the state $(\rho_{1},(u_{1},v_{1}))$. Moreover, the trajectory of the point $O_1$ in the $(u,v)$-plane describes the shock polar $O_{0}P$ as $\alpha$ varies.

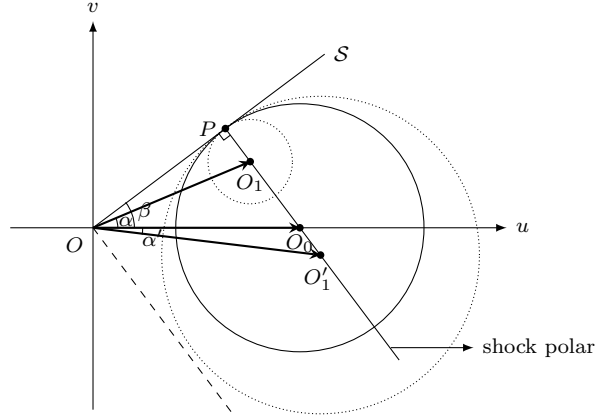
\begin{figure}[H]
	\centering
	\begin{tikzpicture}[scale = 1.1, smooth]
	\draw [-latex] (-1,0)--(5,0) node[right] {\footnotesize$u$};
	\draw [-latex] (0,-2.2)--(0,2.5) node[above] {\footnotesize$v$};
	\node [below left, font=\footnotesize] (0,0) {$O$};
	\draw [thick,-stealth] (0,0)--(2.5,0) node[below=-1pt, font=\footnotesize] {$O_0$};
	\draw (2.5,0) circle (1.5);
	\draw (0,0)--(1.6,1.2)--(2.8,2.1) node[right, font=\footnotesize] {$\mathcal{S}$};
	\draw (1.6,1.2)--(2.5,0)--(3.7,-1.6);
	\draw (1.52,1.14)--(1.58,1.06)--(1.66,1.12);
	\draw [dashed] plot[domain=0:1.7, samples=50] (\x, -4/3*\x);
	\draw [thick,-stealth] (0,0)--(1.9,0.8) node[below=1pt, font=\footnotesize] {$O_1$};
	\draw [densely dotted](1.9,0.8) circle (0.51);
	\draw [thick,-stealth] (0,0)--(2.75,-0.33) node[left=1pt, below=1pt, font=\footnotesize] {$O_1'$};
	\draw [densely dotted](2.75,-0.33) circle (1.92);
	\draw (0:0.3) arc (0:22:0.3);
	\node [right, font=\footnotesize] at (0.2,0.08) {$\alpha$};
	\draw (0:0.6) arc (0:-8:0.6);
	\node [right, font=\footnotesize] at (0.48,-0.12) {$\alpha'$};
	\draw (0:0.5) arc (0:38:0.5);
	\node [right, font=\footnotesize] at (0.43,0.2) {$\beta$};
	\fill(2.5,0)circle(1.3pt);
	\fill(1.9,0.8)circle(1.3pt);
	\fill(2.75,-0.33)circle(1.3pt);
	\fill(1.6,1.2)circle(1.3pt);
	\node [left, font=\footnotesize] at (1.6,1.2) {$P$};
	\draw [-latex](3.6,-1.45)--(4.6,-1.45) node[right] {\footnotesize{shock polar}};
	\end{tikzpicture}
	\caption{Shock polar for a Chaplygin gas.}
	\label{fg3}
\end{figure}

Similarly, we assume that a rarefaction wave $\mathcal{S}$ in the $(x_{1},x_{2})$-plane is straight and passes through the origin $O$, and moreover the oncoming flow state $\big(\rho_{0},(u_{0},0)\big)$ and  the outgoing flow state $\big(\rho'_{1},(u'_{1},v'_{1})\big)$ on the both sides are constant (see \Cref{fg5}). For a given state $(\rho_{0},(u_{0},0))$ with $u_{0}>c_{0}$, using the argument as above, we have
\begin{equation*}
	\begin{aligned}
		c'_{1}&=\dfrac{c_0-\tan\alpha'\sqrt{u^2_0-c^2_0}}{c_0\tan\alpha'+\sqrt{u^2_0-c^2_0}}\sqrt{u^{2}_{0}-c^{2}_{0}},\\
		u'_{1}&=\frac{u_{0}\sqrt{u^{2}_{0}-c^{2}_{0}}}{\sqrt{u^{2}_{0}-c^{2}_{0}}+c_{0}\tan\alpha'},\\
		v'_{1}&=\frac{u_{0}\tan\alpha'\sqrt{u^{2}_{0}-c^{2}_{0}}}{\sqrt{u^{2}_{0}-c^{2}_{0}}+c_{0}\tan\alpha'},
	\end{aligned}
\end{equation*}
where $\alpha'<0$ is the angle between the oncoming flow and the outgoing flow.

In conclusion, the shock polar for a Chaplygin gas is a half-line, extending infinitely from the tangent point $P$ and always perpendicular to the pressure wave $\mathcal{S}$ (see \Cref{fg3}).

It is well known that there may occur a phenomenon of concentration or cavitation for a Chaplygin gas. Now, using the shock polar discussed above, we impose some restriction on the oncoming flow to avoid these phenomena.

From \Cref{fg3}, we see that as $\alpha\rightarrow\beta$, the sound speed  $c_{1}\rightarrow 0$. Then the angle $\beta$ must be greater than $\alpha$, i.e., $\sin\beta>\sin\alpha$.  Using this inequality and  \eqref{eq:A.2}, we have $u_{0}<{c_{0}}/{\sin\alpha}$.  This condition means that for a fixed $\alpha$, if the oncoming flow passes the wing too quickly, that is, $u_{0}\geq{c_{0}}/{\sin\alpha}$, then the flow between the shock and the wedge will concentrate at once. Such a phenomenon is called concentration (see \cite{Bren05}). To avoid this phenomenon, the oncoming flow should satisfy the condition
\begin{equation}\label{eq:A.4}
	c_{0}<u_{0}<\frac{c_{0}}{\sin\alpha}.
\end{equation}	

Also, \Cref{fg3} shows that the sound speed $c'_{1}=\sqrt{A}/{\rho'_{1}}\rightarrow +\infty$ as $\alpha'\rightarrow\beta-{\pi}/{2}$. In other words, if the angle between the rarefaction wave and the velocity of the outgoing flow approaches $\pi/2$, then a phenomenon of cavitation occurs. Hence the angle $\alpha'$ should be greater than $\beta-{\pi}/{2}$, i.e., $\sin\beta<\cos\alpha'$. Then, it follows from this condition and \eqref{eq:A.2} that
\begin{equation}\label{eq:A.5}
	u_{0}>\frac{c_{0}}{\cos\alpha'}.
\end{equation}

\section{Mach cones in 3-D potential flow}\label{sec:appendix b} For the reader's convenience, we calculate the explicit expression of Mach cones for the three dimensional potential equation (also see Lemma 1.1 in \cite{CY15}). From \eqref{eq:2.31}, we have
\begin{equation}\label{eq:B.1}
	\begin{cases}
		|\nabla_{\bm{x}}\Phi\cdot\bm{\zeta}(\tau)|=c,\\
		\bm{x}\cdot\bm{\zeta}(\tau)=0,\\
		\bm{x}\cdot\bm{\zeta}'(\tau)=0,
	\end{cases}
\end{equation}
where $|\bm{\zeta}(\tau)|=1$ and $\tau\in[0,2\pi)$. By eliminating the parameter $\tau$, we reduce \eqref{eq:B.1} to the form
\begin{multline}\label{eq:B.2}
	\big((q^{2}-v^{2}_{1})x_{1}-v_{1}v_{2}x_{2}-v_{1}v_{3}x_{3}\big)^{2}+q^{2}(v_{3}x_{2}-v_{2}x_{3})^{2}\\=\frac{c^{2}(q^{2}-v^{2}_{1})}{q^{2}-c^{2}}\Big(v_{1}x_{1}+v_{2}x_{2}+v_{3}x_{3}\Big)^{2}
\end{multline}
with $\nabla_{\bm{x}}\Phi=(v_{1},v_{2},v_{3})$ and $q^{2}=v^{2}_{1}+v^{2}_{2}+v^{2}_{3}$. A tedious computation shows that the left-hand side of \eqref{eq:B.2} can be rewritten as
\begin{equation*}
	(q^{2}-v^{2}_{1})\big(q^{2}(x^{2}_{1}+x^{2}_{2}+x^{2}_{3})-(v_{1}x_{1}+v_{2}x_{2}+v_{3}x_{3})^{2}\big).
\end{equation*}
This further reduce \eqref{eq:B.2} to the form
\begin{align*}
	(q^{2}-c^{2})(x^{2}_{1}+x^{2}_{2}+x^{2}_{3})=(v_{1}x_{1}+v_{2}x_{2}+v_{3}x_{3})^{2},
\end{align*}
or equivalently,
\begin{equation}\label{eq:B.5}
	|\nabla_{\bm{x}}\Phi\cdot\bm{x}|^{2}-(|\nabla_{\bm{x}}\Phi|^{2}-c^{2})|\bm{x}|^{2}=0
\end{equation}
by the definition of $\Phi$. Hence,
\begin{equation}\label{eq:B.6}
	|\nabla_{\bm{x}}\Phi|^{2}-\vert \nabla_{\bm{x}}\Phi\cdot\frac{\bm{x}}{|\bm{x}|}\vert^{2}=c^{2}\quad\text{for}~\bm{x}\in\mathbb{R}^{3}\setminus\{0\}.
\end{equation}

\section*{Acknowledgments} The authors are indebted to Professor Shuxing Chen for many valuable comments. Bingsong Long also acknowledges the Center for Mathematical Sciences of HUST for the invitation and hospitality.


\end{document}